\documentclass{article}
\usepackage{comment}
\usepackage[T2A]{fontenc}
\usepackage[utf8]{inputenc}
\usepackage[english]{babel}
\usepackage[tbtags]{amsmath}
\usepackage{amsfonts,amssymb,mathrsfs,amscd}
\usepackage[hyper]{msb-a}

\makeatletter
\gdef\No{{\select@language{english}\textnumero}}
\makeatother

\JournalName{}
\numberwithin{equation}{section}
\theoremstyle{plain}
\newtheorem{theorem}{Theorem}
\newtheorem{lemma}{Lemma}[section]

\theoremstyle{definition}

\newtheorem{proof}{Proof}

\begin{document}

\title{ On short Edgeworth expansions for
weighted sums of random vectors}
\author[S.\,A.~Ayvazyan]{S.\,A.~Ayvazyan}
\address{Lomonosov Moscow State University.}
\email{ayvazyansky@yandex.ru}

\date{}
\udk{}

\maketitle

\begin{fulltext}

\begin{abstract}
The "typical"$\;$asymptotic behavior of the weighted sums of independent, identically distibuted random vectors in $k$-dimensional space is considered. It is shown that  under finitnes of fifth absolute moment of an individual term the rate of convergence by Edgeworth correction in the multivariate central limit theorem is of order $O(1/n^{3/2})$. This extends the one-dimensional Bobkov(2020) result. 
\end{abstract}


\markright{ On short Edgeworth correct. for weighted sum of random vectors}

%\footnotetext[0]{Работа выполнена при поддержке РФФИ (грант \No~08-01-00317) и
%Программы поддержки ведущих научных школ РФ (грант \No~НШ-3906.2008.1).}

\section {Introduction and Main Result} \label {1}

Let $ X_1, X_2, \dots, X_n$ be independent, identically distributed random vectors in $\mathbb{R}^k$ with finite third absolute moment  $\beta_3=\mathbb{E}\|X_1\|^3< \infty$, zero mean $\mathbb{E}X_1={0}$ and unit covariance matrix $ \mathrm{Cov}(X_1)=I$.  
Let $\phi(x)$ be the density function of standard Gaussian random vector in $\mathbb{R}^k$ with zero mean and unit covariance matrix.
Denote by $ \mathfrak{B}$, the class of all Borel convex sets in $\mathbb{R}^k$. 
Sazonov~\cite{Saz} obtained the following error bound of approximation for distribution of the normalized sum of random vectors by the standard multivariate normal law:  
 \begin{equation}\label{d1}\sup\limits_{B \in  \mathfrak{B}}\left|\mathbb{P}\left(\frac{1}{\sqrt n}\sum_{i=1}^{n} X_i \in B \right)-\int_{B}\phi(x)dx) \right|\leq C(k)\frac{\beta_3}{\sqrt n},
 \end{equation} 
where $C(k)$ depends on dimension $k$ only. 

The bound (\ref{d1})  is optimal one in general. Moreover, the rate $O(1/\sqrt{n})$ can not be improved under higher order moment assumptions. This is easy to show in one-dimensional case $k=1$ taking $X$ such that $$
\mathbb{P}(X_1=1) = \mathbb{P}(X_1=-1) = 1/2.\nonumber$$

However, the situation is different when we consider a weighted sum
$$
\theta_1X_1 + \dots + \theta_nX_n,
$$
where $\sum_{j=1}^n\theta_j^2=1$. If we are interested in  
the typical behavior of these sums for most of $\theta$ in the sense of the normalized Lebesgue measure $\lambda_{n-1}$ on the unit sphere $$S^{n-1}=\{(\theta_1,\dots,\theta_n): \sum_{j=1}^n\theta_j^2=1\},$$ then we have to refer to a recent remarkable result due to Klartag and Sodin \cite{KS}. In \cite{KS} they  have showed that in one-dimensional case $k=1$ for any $\rho: 1>\rho>0,$ there exists a set $\mathbb {Q} \subseteq S^{n-1}: \lambda_{n-1}(\mathbb {Q}) > 1 - \rho,$ and a constant $C(\rho)$ depending on $\rho$ only such that for any $\theta=(\theta_1,\dots,\theta_n) \in \mathbb {Q}$ one has
\begin{equation}\label{d2}
\sup\limits_{a,b \in \mathbb{R}, a<b}\left|\mathbb{P}\left(a \leq \sum_{i=1}^{n}\theta_iX_i \leq b \right)-\int_a^b \phi(x)dx \right|\leq C(\rho)\frac{\beta_4}{n},
\end{equation}
where $\beta_4=\mathbb{E}|X_1|^4$ and $C(\rho) \leq C \log^2(\frac{1}{\rho}) $ with some absolute constant $C$.   It is clear that $C(\rho) \to \infty $ as $\rho$ tends to $0$. And the case of equal weights, that is when $\theta_i = 1/\sqrt{n}$ for all $i = 1, \dots , n$ is the worst case in the sense of closeness of distribution function  of weighted sum to the standard normal distribution function. The result was extended to multidimensional case in \cite{Ay}.

For independent, identically distributed random variables Bobkov \cite{Bob} refined the rates of approximation for distributions of weighted sums (\ref{d2}) up to order $O(n^{-3/2})$ by using the Edgeworth correction of the fourth order provided $\beta_5=\mathbb{E}|X_1|^5 < \infty$. 
The following corrected normal "distribution"$\;$ function was introduced:
$$
G(x)=\Phi(x)-\frac{\beta_4-3}{8n}(x^3-3x)\phi(x),
$$
where  $\Phi(x)=\int_{-\infty}^{x}\phi(y)dy$. Using $\mathbb{E}_\theta$ to denote the integral over the normalized Lebesgue measure on the unit sphere $S^{n-1}$, Bobkov proved, that if the third moment $\mathbb {E} X_1^3=0$,  then the rate of convergence by $G(x)$ can be improved
\begin{equation}\label{d3}
\mathbb{E}_\theta \sup_x\Big| P\Big(\sum_{i=1}^{n}\theta_iX_i\leq x\Big)-G(x)\Big| \leq C \frac {\beta_5} {n^{3/2}}, 
\end{equation}
where $ C $ is a universal constant. 

In multidimensional case for a given positive integer vector $\alpha$ we define $\alpha$-moment of random vector $X_1$ as $\mu_\alpha=\mathbb{E}X_1^{\alpha}$ and  determine density $g(x)$ of corrected  normal "distribution":
$$\phi(x)+\frac{3}{n}\phi(x)\Big(-\frac{1}{24}\Big[(\mu_{4,0,..,0}-3)(3-6x_1^2+x_1^4)
+\dots+(\mu_{0,0,..,4}-3)(3-6x_k^2+x_k^4)\Big]
$$
$$
-\frac{1}{6}\Big[\mu_{3,1,..,0}(x_1^3x_2-3x_1x_2)+\dots+\mu_{0,..,1,3}(x_k^3x_{k-1}-3x_kx_{k-1})\Big]
$$ 
$$
-\frac{1}{4}\Big[(\mu_{2,2,..,0}-1)(1-x_1^2-x_2^2+x_1^2x_2^2)+\dots+(\mu_{0,..,2,2}-1)(1-x_k^2-x_{k-1}^2+x_{k-1}^2x_k^2)\Big]
$$ 
$$
-\frac{1}{2}\Big[\mu_{2,1,1,..,0}(x_1^2x_2x_3-x_2x_3)+\dots+\mu_{0,..,1,1,2}(x_k^2x_{k-1}x_{k-2}-x_{k-1}x_{k-2})\Big]
$$
$$
-\Big[\mu_{1,1,1,1,..,0}(x_1x_2x_3x_4)+\dots+\mu_{0,..,1,1,1,1}(x_{k-3}x_{k-2}x_{k-1}x_{k})\Big]\Big).
$$
This work examines the rate of convergence of a weighted sum of independent, identically distributed random vectors to a multivariate standard normal vector by Edgeworth correction. The main achievement of the work is the extension of Bobkov's result (\ref{d3}) to the multidimensional case using the signed measure $G(B)=\int_{B}g(x)dx$, presented in the following theorem:
 
\begin {theorem} \label {t3}% immediately specify the link
Let $ X_1, X_2,\dots,X_n $ be independent, identically distributed random vectors of dimension $ k $ with zero mean $ \mathbb {E} X_1 = \overline {0} $, unit covariance matrix $ \mathrm{Cov}(X_1) = I $ and finite fifth absolute moments $\beta_5=\mathbb {E} \| X_1 \| ^ 5 $. Denote by $ \mathfrak {B} $   the class of all convex Borel sets.  If for any positive integer vector $\alpha$ with $|\alpha|=3$ holds $\mathbb {E}X_1^\alpha=0$, then one has
\begin{equation}\label{d0}
    \mathbb {E}_\theta\sup_{B \in \mathfrak {B}}\Big|P\Big(\sum_{i=1}^{n}\theta_iX_i \in B\Big)-G(B)\Big| \leq C (k) \frac {\beta_5} {n^{3/2}}, 
 \end{equation}
where $ C (k) $ is a universal constant depending only on the dimension $ k. $
\end {theorem}

\section {Notation and auxiliary results} \label {2}
{\it Notation}
In the following the weighting coefficients $\theta_1,\theta_2, \dots, \theta_n$  belong to the unit sphere $S^{n-1} = \{\theta_1,\theta_2 , \dots , \theta_n: \sum_{j = 1 } ^ n \theta_j ^ 2 = 1 \}.$ For a  given  random vector $Z$ in $R^k$, a positive integer vector $\alpha$ and a positive integer $s$, we define the $\alpha$-moment as $\mu_\alpha(Z)=\mathbb{E }Z^\alpha$ and the absolute moment of order $s$ as $\rho_s(Z)=\mathbb{E}\|Z\|^s$.
Consider a sequence of independent, not necessarily identically distributed, random vectors $X_1,X_2,\dots,X_n$ with zero mean and unit covariance matrix. And introduce the following truncated random vectors $Y_j$ and $Z_j$ as
\begin{equation} \label {d8}
Y_j = X_j \mathbb{I} (\| \theta_j X_j \| \leq 1),\;\;
Z_j = Y_j- \mathbb {E} Y_j,\;j=1\dots n,
\end{equation}
where $\mathbb{I}(A)$ is the indicator function of event $A$.
Truncated random vectors have an absolute moment of any order, however, their moments may not coincide with the moments of $X_1, X_2,\dots, X_n$. Therefore, we introduce 
\begin{equation} \label {d10}
A_n = \sum_{j = 1} ^{n} \theta_j \mathbb {E} Y_j,\;\;
D = \sum_{j = 1} ^{n} \theta_j ^ 2 \mathrm{Cov}(Z_j).
\end{equation}
Also below we denote for $j=1\dots n$, $|\alpha|\geq 1$ and $s\geq 1$
\begin{equation} \label {d4}
\mu_{\alpha,j} = \mu_\alpha(X_j),\;\;\rho_{s,j} = \rho_s(X_j),
\end{equation}
\begin{equation} \label {d11}
\tilde{\mu}_{\alpha,j} = \mu_\alpha(Z_j),\;\;\tilde{\rho}_{s,j} = \rho_s(Z_j),
\end{equation}
\begin{equation} \label {d6}
\rho_s = \sum_{j = 1} ^{n} \rho_{s}(\theta_iX_i),\;\;\beta_s = \frac {1} {n} \sum_{j = 1} ^ n \rho_ {s}(X_i),
\end{equation}
\begin{equation} \label {d12}
\tilde{\rho}_s = \sum_{j = 1}^n \rho_s (\theta_jZ_j), \;\;\eta_s = \sum_{j = 1}^n \rho_s (Q \theta_jZ_j),
\end{equation}
where matrix $Q^2=D^{-1}$.

To derive the result (\ref{d0}), the author is using the technique of proving the multivariate central limit theorem from the monograph by Bhattacharya and Ranga Rao \cite{B}, as well as the further development of methods and approaches proposed by Klartagam, Sodin \cite{KS} and Bobkov \cite{Bob}. The proof relies on the use of characteristic functions. This method is based on the Fourier transform and its extension, the Fourier-Stieltjes transform. The characteristic function of a random vector $Z$ is defined as
$$\varphi(t)= \mathbb {E} \exp (i \langle t,Z\rangle)
,$$
where $i$ is a unit imaginary number
Next, we define the logarithm of a non-zero complex number $z = r\exp(i\xi)$ as
$$
\log(z) = \log(r) + i \xi, \; r>0,\; \; \xi \in (-\pi,\pi ],
$$
therefore, we always take the so-called main branch of the logarithm. Let $Z$ have an absolute moment of order $m$, then in the neighborhood of zero the logarithm of the characteristic function expands into a Taylor series
\begin{equation} \label {d17_0}
\log (\varphi(t)) = \sum_ {1 \leq | \nu | \leq m} \kappa_ \nu (Z) \frac {(it) ^ \nu} {\nu!} + o (\| t \| ^ m),
\end{equation}
where $\kappa_\nu (Z)$ is the cumulant of the random vector $Z$ of order $\nu$. The following formal identity allows us to express cumulants in terms of
corresponding moments by equating the coefficients of $t^\nu$ on both sides (see Chapter 2, Section 6 in \cite{B})
\begin{equation} \label {d17}
\sum_{\nu\geq 1}\kappa_{\nu}(Z)\frac{(it)^\nu}{\nu!}=\sum_{i=1}^{\infty}(-1) ^{s+1}\frac{1}{s}\Big(\sum_{\nu\geq1}\mu_{\nu }(Z)\frac{ (it)^\nu }{\nu!}\Big)^s.
\end{equation}
From the identity (\ref{d17}) and the inequality $|\mu_\nu(Z)|\leq \rho_{|\nu|}(Z)$, it follows that for the absolute value of the cumulants one has
\begin{equation} \label {d18}
| \kappa_{\nu}(Z) | \leq c (\nu) \rho_{| \nu |}(Z),
\end{equation}
where $c (\nu)$ is a constant that depends only on the vector $\nu$. Let us introduce the following notation for $j=1\dots n,$ $|\nu|>1$
$$\tilde{\kappa}_{\nu,j}=\kappa_\nu(Z_j),\;\; \kappa_{\nu,j}=\kappa_\nu(X_j),$$
\begin{equation}\label{d20}
\tilde{\kappa}_ \nu = \kappa_ \nu \Big(\sum_ {j = 1} ^ n \theta_jZ_j\Big),\;\;
\kappa_ \nu = \kappa_ \nu \Big(\sum_ {j = 1} ^ n \theta_jX_j\Big).
\end{equation}
Since we are considering independent random variables, the characteristic function of the sum $\sum_{j = 1} ^ n \theta_jZ_j$ is calculated as
\begin{equation} \label {d17_1}
\hat{F}_Z(t)=\prod_{j=1}^n\varphi_j (\theta_j t),
\end{equation}
where
\begin{equation} \label {d17_2}
\varphi_j (t) = \mathbb {E} \exp (i \langle t,Z_j\rangle), \; j = 1,\dots,n.
\end{equation}
If we take $Z=\sum_{j = 1} ^ n \theta_jZ_j$, then from (\ref{d17_0}) and (\ref{d17_1}) an important property of the cumulants of the sum of random vectors follows (similarly, if $Z =\sum_{j = 1}^n\theta_jX_j$)
\begin{equation} \label {d19}
\tilde{\kappa}_\nu = \sum_{j = 1}^n \theta_j^{| \nu |} \tilde{\kappa}_{\nu,j}\;\;\Big(\kappa_\nu = \sum_ {j = 1} ^ n \theta_j ^ {| \nu |} \kappa_{ \nu,j}\Big).
\end{equation}

  We need to introduce a multivariate normal distribution $\Phi_{a, V}$ with mean $a$ and covariance matrix $V$, whose density is given by
\begin{equation} \label {d21}
\phi_{a,V}(x)=(2\pi)^{-\frac{k}{2}}(\mathrm{det}V)^{-\frac{1}{2}}\exp \Big(-\frac{1}{2}\langle x-a,V^{-1}(x-a)\rangle\Big),
\end{equation}
where $a \in R^k$ and $V$ is a symmetric positive definite matrix $k\times k$. If $a=\overline{0}$ and $V=I$, then by $\Phi(x)$ and $\phi(x)$ we mean $\Phi_{\overline{0},I}(x)$ and $\phi_{\overline{0},I}(x)$, respectively. The characteristic function of the normal distribution is known to be equal to
\begin{equation} \label {d22}
\hat {\phi}_{a, V}(t)=\exp\Big(i\langle t,a\rangle-\frac{1}{2}\langle t,Vt\rangle \Big).
\end{equation}

For improving the rate of convergence in (\ref{d1})  the distribution
of the normal law should be corrected in some smooth way.
In order to do this, we use the expansion of the characteristic function in terms of cumulants and  to recover the density through the inverse Fourier transform. For a random vector $Z$, we introduce the polynomials $\hat{P}_r (t: \{\kappa_\nu(Z)\})\;,r\geq0,$ from the formal expression
$$
1+ \sum_ {r = 1} ^ \infty \hat{P}_r (t: \{\kappa_\nu(Z)\}) u ^ r = \exp \Big (\sum_ {|\nu| = 3} ^ \infty \frac{\kappa_\nu(Z)}{\nu!} u ^ {|\nu|-2} \Big),
$$
   explicitly we obtain that $\hat{P}_0 (t: \{ \kappa_ \nu(Z) \}) = 1$, and for $r>0\; \hat{P}_r (t: \{ \kappa_\nu(Z) \} )$ is calculated as
\begin{equation} \label {d23}
  \sum_{m = 1}^r \frac{1}{m!} \sum_{i_1,\dots,i_m: \sum_{j = 1} ^ mi_j = r } \Bigg[\sum_{\nu_1 ,\dots, \nu_m: | \nu_j | = i_j + 2} \frac{\kappa_{\nu_1}(Z) \dots\kappa_{\nu_m}(Z) } {\nu_1! \dots \nu_m!} \Bigg] t^{\nu_1 + \dots + \nu_m}.
\end{equation}
If we assume that $\mathbb {E}Z=\overline{0}$ and $\mathrm{Cov}(Z)=V$, then, as shown in Lemma 7.2, \cite{B} the function $\hat{ P}_r (it: \{ \kappa_\nu(Z) \} )\hat {\phi}_{\overline{0},V}(t) $ is the Fourier transform of
\begin{equation} \label {d24}
P_r (-\phi_{\overline{0},V}, \{ \kappa_\nu(Z) \} ),
\end{equation}
 obtained by formally substituting
$(-1)^{|\nu|}D^\nu\phi$ by $(it)^\nu$ for each $\nu$ in the polynomials $\hat{P}_r (it : \{ \kappa_ \nu(Z) \} )$, where $ D ^ \nu $ is a differential operator. Let us define the finite signed measure
\begin{equation} \label {d25}
P_r (-\Phi_{\overline{0},V}, \{ \kappa_\nu(Z) \} ),
\end{equation}
whose density is
$P_r (-\phi_{\overline{0},V}, \{ \kappa_\nu(Z) \} )$.
Putting the sum of random vectors $\sum_{i=1}^n\theta_iX_i$ or $\sum_{i=1}^n\theta_iZ_i$ in (\ref{d25}) as the  random vector $Z$, we obtain the Edgeworth expansions (signed measures) of order $s\geq1$ of these sums, respectively:
\begin{equation}\label{d25_1}
\Psi_s(B)=\sum_{r=0}^{s}P_r (-\Phi: \{ \kappa_\nu \} )(B),\;\;
\Psi_s^{\prime}(B)=\sum_{r=0}^{s}P_r (-\Phi_{\overline{0},D}: \{ \tilde{\kappa}_\nu \} )(B).
\end{equation}

Now, we present some lemmas used in the proof of the main result.
\begin {lemma} \label {lemma1} 
Let $X_1$ be a random vector with  $\mathbb{E}\|X_1\|^4< \infty$, zero mean and unit covariance matrix.
If for every positive integer vector $\alpha$ with $|\alpha|=3$ holds $\mathbb {E}X_1^\alpha=0$, then 
$$
P_1 (-\phi: \{ \kappa_\nu(X_1) \} )+
P_2 (-\phi: \{ \kappa_\nu(X_1) \} )
$$
$$
=\phi(x)\Big(-\frac{1}{24}\Big[(\mu_{4,0,..,0}-3)(3-6x_1^2+x_1^4)
+\dots+(\mu_{0,0,..,4}-3)(3-6x_k^2+x_k^4)\Big]
$$
$$
-\frac{1}{6}\Big[\mu_{3,1,..,0}(x_1^3x_2-3x_1x_2)+\dots+\mu_{0,..,1,3}(x_k^3x_{k-1}-3x_kx_{k-1})\Big]
$$ 
$$
-\frac{1}{4}\Big[(\mu_{2,2,..,0}-1)(1-x_1^2-x_2^2+x_1^2x_2^2)+\dots+(\mu_{0,..,2,2}-1)(1-x_k^2-x_{k-1}^2+x_{k-1}^2x_k^2)\Big]
$$ 
$$
-\frac{1}{2}\Big[\mu_{2,1,1,..,0}(x_1^2x_2x_3-x_2x_3)+\dots+\mu_{0,..,1,1,2}(x_k^2x_{k-1}x_{k-2}-x_{k-1}x_{k-2})\Big]
$$
\begin {equation} \label {d25_2}
-\Big[\mu_{1,1,1,1,..,0}(x_1x_2x_3x_4)+\dots+\mu_{0,..,1,1,1,1}(x_{k-3}x_{k-2}x_{k-1}x_{k})\Big]\Big),
\end {equation} 
where $\mu_\nu=\mathbb{E}X_1^\nu$, $|\nu|>0$, $P_r (-\phi: \{ \kappa_{\nu,1} \} )$ is defined in (\ref{d24}) and (\ref{d25_1}), $r\geq0$.
 \end {lemma}
 \begin{proof}
 Since we suppose that $\mu_\nu=0$ for $|\nu|=3$, $\mathbb {E}X_1=\overline{0}$ and $\mathrm{Cov}(X_1)=I$, and due to (\ref{d18}), we get
\begin {equation} \label {d26}
\kappa_\nu(X_1)=\mu_\nu=0, |\nu|=3,
\end {equation} 
and 
\begin {equation} \label {d27}
\kappa_\nu(X_1)=\mu_\nu=1,|\nu|=2,
\end {equation}
in case when $|\nu|=4$, again using formal identity (\ref{d17})  and (\ref{d26}), (\ref{d27}),  we conclude the following
$$
\kappa_{4,..,0}(X_1)=\mu_{4,..,0}-3\mu_{2,..,0}^2=\mu_{4,..,0}-3,
$$
$$
\kappa_{0,..,4,..,0}(X_1)=\mu_{0,..,4,..,0}-3\mu_{0,..,2,..,0}^2=\mu_{0,..,4,..,0}-3,
$$
$$\dots$$
$$
\kappa_{3,1,..,0}(X_1)=\mu_{3,1,..,0},
$$
$$\dots$$
$$
\kappa_{0,..,3,1,..,0}(X_1)=\mu_{0,..,3,1,..,0},
$$
$$\dots$$
$$
\kappa_{2,2,..,0}(X_1)=\mu_{2,2,..,0}-\mu_{2,0,..,0}\cdot\mu_{0,2,..,0}=\mu_{2,2,..,0}-1,
$$
$$\dots$$
$$
\kappa_{0,..,2,2,..,0}(X_1)=\mu_{0,..,2,2,..,0}-\mu_{0,..,2,0,..,0}\cdot\mu_{0,..,0,2,..,0}=\mu_{0,..,2,2,..,0}-1,
$$
$$\dots$$
$$
\kappa_{2,1,1,..,0}(X_1)=\mu_{2,1,1,..,0},
$$
$$\dots$$
$$
\kappa_{0,..,2,1,1,..,0}(X_1)=\mu_{0,..,2,1,1,..,0},
$$
$$\dots$$
$$
\kappa_{1,1,1,1,..,0}(X_1)=\mu_{1,1,1,1,..,0},
$$
$$\dots$$
\begin {equation} \label {d28}
\kappa_{0,..,1,1,1,1,..,0}(X_1)=\mu_{0,..,1,1,1,1,..,0}.
\end {equation} 
$$\dots$$
We consider the first polynomial expression, by (\ref{d23}) and (\ref{d26}), one has
\begin {equation} \label {d29}
\hat{P}_1 (t, \{ \kappa_\nu(X_1) \} )=\frac{\kappa_3(X_1,t)}{3!}=\sum_{|\nu|=3}\kappa_\nu(X_1)\frac{t^\nu}{\nu!}=\sum_{|\nu|=3}\mu_\nu\frac{t^\nu}{\nu!}=0
\end {equation} 
which means that the polynomial also $P_1 (-\phi, \{ \kappa_\nu  \} )(x)=0.$

next for the second polynomial, again using (\ref{d23}) and (\ref{d29}), we have
$$
\hat{P}_2 (t, \{ \kappa_\nu(X_1)  \} )=\frac{\kappa_4(X_1,t)}{4!}+\frac{1}{2!}\Big(\frac{\kappa_3(X_1,t)}{3!}\Big)^2
$$
$$
=\sum_{|\nu|=4}\kappa_\nu(X_1)\frac{t^\nu}{\nu!}+\frac{1}{2!}\Big(\frac{\kappa_3(X_1,t)}{3!}\Big)^2
=\sum_{|\nu|=4}\kappa_\nu(X_1)\frac{t^\nu}{\nu!}.
$$
 Due to Lemma 6.2 \cite{B}, one has
%$$P_1 (x, \{ \kappa_\nu  \} )=-\sum_{|\nu|=3}\kappa_\nu\frac{(-1)^{|\nu|}(D^\nu\phi_{\overline{0},V})(x)}{\nu!}=0$$$$=-\phi(x)\Big(1+\frac{1}{6}\Big[\kappa_{3,0,..,0}(3x_1-x_1^3)+\dots+\kappa_{0,0,..,3}(3x_k-x_k^3)\Big]$$$$+\frac{1}{2}\Big[\kappa_{2,1,..,0}(x_2-x_1^2x_2)+\dots+\kappa_{0,..,1,2}(x_{k-1}-x_{k-1}x_{k}^2)\Big]$$$$+\Big[\kappa_{1,1,1,..,0}(-x_1x_2x_3)+\dots+\kappa_{0,..,1,1,1}(-x_kx_{k-1}x_{k-2})\Big]\Big)$$ 
$$P_2 (-\phi, \{ \kappa_\nu(X_1)  \} )(x)=-\sum_{|\nu|=4}\kappa_\nu(X_1)\frac{(-1)^{|\nu|}(D^\nu\phi)(x)}{\nu!}$$
$$=-\phi(x)\Big(\frac{1}{24}\Big[\kappa_{4,0,..,0}(X_1)\cdot(3-6x_1^2+x_1^4)+\dots+\kappa_{0,0,..,4}(X_1)\cdot(3-6x_k^2+x_k^4)\Big]
$$
$$
+\frac{1}{6}\Big[\kappa_{3,1,..,0}(X_1)\cdot(x_1^3x_2-3x_1x_2)+\dots+\kappa_{0,..,1,3}(X_1)\cdot(x_k^3x_{k-1}-3x_kx_{k-1})\Big]
$$ 
$$
+\frac{1}{4}\Big[\kappa_{2,2,..,0}(X_1)\cdot(1-x_1^2-x_2^2+x_1^2x_2^2)+\dots+\kappa_{0,..,2,2}(X_1)\cdot(1-x_k^2-x_{k-1}^2+x_{k-1}^2x_k^2)\Big]
$$ 
$$
+\frac{1}{2}\Big[\kappa_{2,1,1,..,0}(X_1)\cdot(x_1^2x_2x_3-x_2x_3)+\dots+\kappa_{0,..,1,1,2}(X_1)\cdot(x_k^2x_{k-1}x_{k-2}-x_{k-1}x_{k-2})\Big]
$$
\begin{equation}\label{d30}
+\Big[\kappa_{1,1,1,1,..,0}(X_1)\cdot(x_1x_2x_3x_4)+\dots+\kappa_{0,..,1,1,1,1}(X_1)\cdot(x_{k-3}x_{k-2}x_{k-1}x_{k})\Big]\Big),
\end{equation}
finally, it remains to substitute the values of cumulants with expressions in terms of moments (\ref{d28}) to get (\ref{d25_2}).
\end{proof}
Below in lemmas \ref{lemma2}-\ref{lemma7} we assume that $ X_1, X_2,\dots,X_n $ are independent random vectors with zero mean, unit covariance matrix and finite absolute moment of order $s>2$.
 \begin {lemma} \label {lemma2}
Let $\rho_s \leq (8k)^{-1},$
then the covariance matrix $ D $ satisfies
$$ \Big | \langle t, Dt \rangle - \| t \| ^ 2 \Big | \leq 2k \rho_s \| t \| ^ 2, $$
$$ \| D-I \| \leq \frac {1} {4}, \; \; \; \frac {3} {4} \leq \| D \| \leq \frac {5} {4},\;\;\;\| D ^ {- 1} \| \leq \frac {4} {3}, $$
where $ \rho_s $ and matrix $ D $ are defined in (\ref {d6}) and (\ref {d10}), respectively.
  \end {lemma}
  \begin {proof}
The proof of Lemma follows the scheme of the proof of Corollary 14.2 \cite {B}.
First, we prove two auxiliary inequalities for the means of the original and truncated random vectors
$$ \Big | \mathbb {E} \theta_jY_ {ji} \Big | = \Big | \mathbb {E} X_ {ji} \theta_j\mathbb{I} (\| X_j \theta_j \|> 1) \Big |$$
$$\leq \mathbb {E} \| X_j \theta_j \| \mathbb{I} (\| X_j \theta_j \|> 1) \leq\theta_j ^s\mathbb {E} \| X_j\|^s $$
and
$$ \theta_j ^ 2 \Big | \mathbb {E} X_ {ji} X_ {jl} - \mathbb {E} Y_ {ji} Y_ {jl} \Big | = \theta_j ^ 2 \Big | \mathbb {E} X_ {ji} X_ {jl} \mathbb{I} (\| \theta_jX_ {ji} \|> 1) \Big | $$
$$ \leq \mathbb {E} \| \theta_j X_j \| ^ 2 \mathbb{I} (\| \theta_j X_ {ji} \|> 1) \leq \theta_j ^ s \mathbb {E} \|  X_j \| ^ s. $$
Also, note that
$$ | \mathbb {E} \theta_jY_ {ji} | = | \mathbb {E} \theta_jX_ {ji} \mathbb{I} (\| X_j \theta_j \| \leq 1) | \leq 1. $$
Define  Kronecker delta function  as $\delta_{ij}=\mathbb{I}(i=j),$ for $i,j=1,\dots,k.$
 By definition of covariance matrix  of the weighted sums  $D = \sum_ {j = 1} ^ {n} \theta_j ^ 2\mathrm{Cov}(Z_j)$ (see (\ref{d10})), one has
\begin{equation}\label{d30_1}
 \Big | \langle t, Dt \rangle - \langle t, t \rangle \Big | = \Big | \sum_ {i, l} ^ kt_ {i} t_l (d_ {il} -\delta_ {il}) \Big |, 
 \end{equation}
wherein
$$ \Big | d_ {il} -\delta_ {il} \Big | \leq \sum_ {j = 1} ^ n \theta_j ^ 2 \Big | \mathrm{Cov}(X_ {ji}, X_ {jl}) - \mathrm{Cov}(Y_ {ji}, Y_ {jl}) \Big | $$
$$\leq \sum_ {j = 1} ^ n \theta_j ^ 2 \Big | \mathbb {E} X_ {ji} X_ {jl} - \mathbb {E} Y_ {ji} Y_ {jl} + \mathbb {E} Y_ {ji} \mathbb {E} Y_ {jl} \Big |$$
\begin{equation}\label{d30_2} 
 \leq \sum_ {j = 1} ^ n \Big (|\theta_j| ^ s \mathbb {E} \|  X_j \| ^ s+|\theta_j| ^ s \mathbb {E} \|  X_j \| ^ s \Big) = 2 \rho_s. 
\end{equation}
next, using definition (\ref{d30_1}) and the calculated inequality (\ref{d30_2}), we canclude
\begin{equation}\label{d30_3}  \Big | \langle t, Dt \rangle - \langle t, t \rangle \Big | \leq2 \rho_s \Big (\sum_i ^ k | t_i | \Big) ^ 2 \leq2k \rho_s \| t \| ^ 2. 
\end{equation}
Finally, by the definition of the matrix norm and  (\ref{d30_3}), the  following inequality holds 
$$ \| D-I \| = \sup \limits _ {\| t \| \leq 1} \Big | \langle t, (D-I) t \rangle \Big | \leq2k \rho_s. $$
And since we suppose that $ \rho_s <(8k) ^ {- 1} $, then for the norms of matrices $D-I$ and $D$ hold
\begin{equation}\label{d30_4}
\| D-I \| \leq \frac {1} {4},\;\;\;\frac {3} {4} \leq \| D \| \leq \frac {5} {4}. 
\end{equation}
Further,
$$ \langle t, Dt \rangle \geq \| t \| ^ 2- \frac {1} {4} \| t \| ^ 2 = \frac {3} {4} \| t \| ^ 2. $$
Therefore, the matrix $D$ is not degenerate and there is an inverse matrix $D^{-1}$, for which one has
\begin{equation}\label{d30_5} \| D ^ {- 1} \| \leq \frac {4} {3}.
\end{equation}
\end {proof}

  \begin {lemma} \label {lemma3}
 Let $\alpha$ be a nonnegative integer vector satisfying $1 \leq a \leq s$, then for $j=1\dots n$ ones have
\begin{equation} \label {d38}
 |\mu_\alpha(\theta_jX_j)-\mu_\alpha(\theta_jY_j)|\leq |\theta_j|^s\rho_{s,j},
\end{equation} 
 \begin{equation} \label {d39}
 |\mu_\alpha(\theta_jZ_j)-\mu_\alpha(\theta_jY_j)|\leq |\alpha|(2^{|\alpha|}+1)|\theta_j|^s\rho_{s,j},
\end{equation} 
 where $ \rho_{s,j}$ and random vectors  $Y_j$,\;$Z_j$  are defined in (\ref {d4}) and (\ref{d8}), respectively.
\end{lemma}
\begin {proof}
The proof of Lemma follows the scheme of proofs of  Lemma 14.1 \cite {B}. To prove the first inequality (\ref{d38}) we use the definition of truncated random vector $Y_j=X_j\mathbb{I}(\|\theta_jX_j\|\leq1)$ and the fact $|\mu_{\alpha,j}|\leq\rho_{|\alpha|,j}$
 \begin {equation} \label {d40}
|\mathbb{E}(\theta_jX_j)^\alpha-\mathbb{E}(\theta_jY_j)^\alpha|= |\mathbb{E}(\theta_jX_j)^\alpha\mathbb{I}(\|\theta_jX_j\|\geq1)|
 \end {equation} 
$$
\leq \mathbb{E}\|\theta_jX_j\|^{|\alpha|}\mathbb{I}(\|\theta_jX_j\|\geq1)
\leq \mathbb{E}\|\theta_jX_j\|^s=\rho_s(\theta_jX_j).
$$
To derive the second inequality, for a given vectors $x, y\in R^k$ we use the fact(see (14.12) \cite {B}) that 
 \begin {equation} \label {d41}
|x^\alpha-y^\alpha|\leq |\alpha|(\|x\|^{\|a\|-1}+\|y\|^{\|a\|-1})\max_{1\leq i\leq k}(|x_i-y_i|).
 \end{equation} 
Then, using (\ref{d40}), we conclude that
 \begin {equation} \label {d42}
|\mathbb{E}(\theta_jX_j)-\mathbb{E}(\theta_jY_j)|\leq \rho(\theta_jX_j),
 \end{equation} 
it follows from (\ref{d41}) and (\ref{d42})   that
 \begin {equation} \label {d43}
|\mathbb{E}\theta_jZ_j^\alpha-\mathbb{E}\theta_jY_j^\alpha|\leq|\alpha|\rho_s\cdot(\mathbb{E}\|\theta_jY_j\|^{|\alpha|-1}+\mathbb{E}\|\theta_jZ_j\|^{|\alpha|-1}).
\end {equation} 
Finally

\begin {equation} \label {d44}
\mathbb{E}\|\theta_jY_j\|^{|\alpha|-1}= \mathbb{E}\|\theta_jX_j\|^{|\alpha|-1}\mathbb{I}(\|\theta_jX_j\|\leq1)\leq1,
\end {equation} 
$$
(\mathbb{E}\|\theta_jY_j\|)^2\leq \mathbb{E}\|\theta_jY_j\|^2,
$$
\begin {equation} \label {d45}
\mathbb{E}\|\theta_jZ_j\|^{|\alpha|-1}\leq2^{|\alpha|-1}(\mathbb{E}\|\theta_jY_j\|^{|\alpha|-1}+\mathbb{E}\|\theta_jY_j\|^{|\alpha|-1})\leq 2^{|\alpha|}.
\end {equation} 
Putting (\ref{d44}) and (\ref{d45}) into (\ref{d43}), we get (\ref{d39}).
\end {proof}

\begin {lemma} \label {lemma4}
Let $ \rho_s  \leq (8k)^{-1}$ and $\nu_1,\nu_2,\dots,\nu_m$  be nonnegative vectors such that
$ |\nu_i|\geq3,\;\;1\leq i \leq m,\;\;\; \sum_{i=1}^{m}(|\nu_i|-2)=r,$
for some positive integer $r,\; 1 \le r \le s - 2$, then one has
\begin {equation} \label{d46}
|\tilde{\kappa}_{\nu_1}\dots\tilde{\kappa}_{\nu_m} - \kappa_{\nu_1}\dots\kappa_{\nu_m}|\leq c_2(s,k)\rho_s,
\end {equation}
where  $\kappa_\nu$ and $\tilde{\kappa}_\nu$ are corresponding cumulants defined in  (\ref{d20}) and $\rho_s$ is introduced in (\ref{d6}).
\end{lemma}
\begin {proof}
 The proof of Lemma follows the scheme of the proof of Lemma 14.5 \cite {B}.
 First recall (see (6.14), (6.15) \cite{B}) that for each nonnegative integral
vector $\nu$ cumulant $\kappa_\nu(\theta_jX_j)$ is a linear combination of terms like
$$\theta_j^{|\nu|}\mu_{\alpha_{1,j}}^{r_1}\mu_{\alpha_{2,j}}^{r_2}\dots\mu_{\alpha_{p,j}}^{r_p}$$
where $\alpha_1,\dots,\alpha_p$ are nonnegative integer vectors and $r_1,\dots,r_p$ are positive integers satisfying $\sum_{i=1}^{p}r_i\alpha_i=\nu$. To estimate the difference $\kappa_{\nu,j}-\tilde{\kappa}_{\nu,j}$  it is therefore enough to consider
$$\theta_j^{|\nu|}\mu_{\alpha_{1,j}}^{r_1}\mu_{\alpha_{2,j}}^{r_2}\dots\mu_{\alpha_{p,j}}^{r_p}-\theta_j^{|\nu|}\tilde{\mu}_{\alpha_{1,j}}^{r_1}\tilde{\mu}_{\alpha_{2,j}}^{r_2}\dots\tilde{\mu}_{\alpha_{p,j}}^{r_p},\;\;3\leq|\nu|\leq s.$$
next we can use the (\ref{d39}),   
$|\theta_j^{|\alpha|}\tilde{\mu}_{\alpha,j}-\theta_j^{|\alpha|}\mu_{\alpha,j}|\leq \overline{c}_{1}(\alpha)|\theta_j|^{s}\rho_{s,j}$ from Lemma 4, and by the inequality (\ref{d41}), one has 
$$
\theta_j^{|\nu|}\mu_{\alpha_{1,j}}^{r_1}\mu_{\alpha_{2,j}}^{r_2}\dots\mu_{\alpha_{p,j}}^{r_p}-\theta_j^{|\nu|}\tilde{\mu}_{\alpha_{1,j}}^{r_1}\tilde{\mu}_{\alpha_{1,j}}^{r_2}\dots\tilde{\mu}_{\alpha_{p,j}}^{r_p}
$$
$$
\leq \sum_{i=1}^p r_i\overline{c}_{2}(\alpha_i)\rho_s(\theta_jX_j)\Big(|\mu_{\alpha_i}^{r_1-1}(\theta_jX_j)|+|\tilde{\mu}_{\alpha_i}^{r_1-1}(\theta_jX_j)|\Big)
$$
$$
\times \Big|\mu_{\alpha_{1}}^{r_1}(\theta_jX_j)\dots\mu_{\alpha_{i}}^{r_{i-1,j}}(\theta_jX_j)\tilde{\mu}_{\alpha_{i+1}}^{r_{i+1,j}}(\theta_jX_j)\dots\tilde{\mu}_{\alpha_{p}}^{r_{p}}(\theta_jX_j)  \Big|
$$
\begin {equation} \label{d47}
\leq  \sum_{i=1}^p \overline{c}_3(i,s,r)\rho_{s,j}(\rho_{s}(\theta_jX_j))^{m_i/s}.
\end{equation}
Noting that $m_i=\sum_{i=1}^p r_i|\alpha_i|-|\alpha_i|=|\nu|-|\alpha_i|
$ and since $\rho_s\leq(8k)^{-1}$, from (\ref{d47})  one has
\begin {equation} \label {d48}
|\theta_j|^{|\nu|}|\kappa_{\nu,j}-\kappa_{\nu,j}|\leq |\theta_j|^{s}\overline{c}_{4}(\nu,k)\rho_{s,j},
\end{equation}
summing of the parts (\ref{d48}), we get one 
%of $\kappa_\nu^\prime-\kappa_\nu$   
%summing up terms
\begin {equation} \label {d49}
|\tilde{\kappa}_\nu-\kappa_\nu|\leq \overline{c}_{4}(\nu,k)\rho_{s}.
\end{equation}
Now let $\nu_1,\nu_2,\dots,\nu_m$ be nonnegative vectors satisfying $|\nu_i|\geq3$ and $\sum_{i=1}^{m}(|\nu_i|-2)=r, 1\leq i \leq m.$
By (\ref{d49}) one has
\begin {equation} \label {d50}
|\tilde{\kappa}_{\nu_1}\tilde{\kappa}_{\nu_2}\dots\tilde{\kappa}_{\nu_m}-\kappa_{\nu_1}\kappa_{\nu_2}\dots\kappa_{\nu_{m}}|\leq 
\sum_{i=1}^m\overline{c}_{4}(\nu_i,k)\rho_s|\tilde{\kappa}_{\nu_1}\dots\tilde{\kappa}_{\nu_{i-1}}\kappa_{\nu_{i+1}}\dots\kappa_{\nu_{m}}|.
\end{equation}
Now recall that $\rho_2=k$, then by (\ref{d18}), (\ref{d13}) and using Lemma 1 \cite{Ay}, we have
$$
|\tilde{\kappa}_{\nu_1}\dots\tilde{\kappa}_{\nu_{i-1}}\kappa_{\nu_{i+1}}\dots\kappa_{\nu_{m}}|\leq \overline{c}_{5}(s,k)|\rho_{|\nu_1|}\dots\rho_{|\nu_{i-1}|}\rho_{|\nu_{i+1}|}\dots\rho_{|\nu_{m}|}
$$

$$
\leq \overline{c}_{6}(s,k)\rho_{s}^{(r-|\nu_i|+2)/(s-2)},
$$
and since we suppose that $\rho_s\leq(8k)^{-1},$ then also the following inequality holds
\begin {equation} \label {d51}
|\tilde{\kappa}_{\nu_1}\dots\tilde{\kappa}_{\nu_{i-1}}\kappa_{\nu_{i+1}}\dots\kappa_{\nu_{m}}|\leq \overline{c}_{7}(s,k).
\end{equation}
Putting the expression (\ref{d51}) into (\ref{d50}),
we get the statement of the Lemma (\ref{d46}).
\end {proof}

\begin {lemma} \label {lemma5}
%Summing up terms,
Let $ \rho_s  \leq (8k)^{-1}, $
then for every integer r, $0 \leq r \leq s-2,$ ones have
$$
\Big|P_r(-\phi: \{\kappa_\nu\} )(x) - P_r(-\phi_{\overline{0},D}: \{\tilde{\kappa}_\nu\})(x)\Big|
$$
\begin {equation} \label {d52}
\leq c_3(r,k,s) \rho_{s}(1 +\|x\|^{3r+2})\exp\Big(-\frac{\|x\|^2}{6}+\|x\|\Big),
\end {equation} 
$$
\Big|P_r (-\phi: \{\kappa_\nu\} )(x+A_n) - P_r(-\phi: \{\kappa_\nu\} )(x)\Big|
$$
\begin {equation} \label {d53}
\leq c_4(r,k,s) \rho_{s}(1 +\|x\|^{3r+1})\exp\Big(-\frac{\|x\|^2}{2}+\frac{\|x\|}{8k^{1/2}}\Big),
\end {equation} 
where $A_n$ is defined in (\ref{d10}), $\rho_s$ in (\ref{d6})  and  $P_r(-\phi_{a,V}: \{\kappa_\nu\} )$ in (\ref{d24}).
\end{lemma}
\begin {proof}
The proof of Lemma follows the scheme of the proof of Lemma 14.6 \cite {B}. At first, we have to consider the following auxiliary function 
 $$
 p(x)=1-(\mathrm{det}D)^{-\frac{1}{2}}\exp\Big(\frac{1}{2}\sum_i^kx_i^2 -\frac{1}{2}\sum_{i,j}^kd_{ij}x_{i}x_{j}\Big)
 $$
$$
|p(x)|\leq |(\mathrm{det}D)^{-\frac{1}{2}}-1|\exp\Big(\|D^{-1}-I\|\|x\|^2\Big)
$$
$$
+\frac{1}{2}\|D^{-1}-I\|\|x\|^2\exp\Big(\|D^{-1}-I\|\|x\|^2\Big),
$$
where $d_{ij}$ is $(i,j)$ element of matrix $D^{-1}$. Now we can apply Lemma \ref{lemma2} and get
\begin{equation}\label{d54}
|\langle DD^{-\frac{1}{2}}t,D^{-\frac{1}{2}}t\rangle-\|t\|^2|=|\langle t,D^{-1}t\rangle-\|t\|^2|\leq 2k\rho_s\|D^{-1}\|\|t\|^2\leq k\frac{8}{3}\rho_s\|t\|^2.
\end{equation}
also (\ref{d54}) implies that
\begin{equation}\label{d55}
\|D^{-1}-I\|\leq \frac{1}{3}.
\end{equation}
Then we estimate 
\begin{equation}\label{d56}
|(\mathrm{det}D)^{-\frac{1}{2}}-1|=(\mathrm{det}D)^{\frac{1}{2}}|(\mathrm{det}D)^{-\frac{1}{2}}-1|
\end{equation}
$$
\leq(\mathrm{det}D)^{\frac{1}{2}}|(\mathrm{det}D)^{\frac{1}{2}}-1||(\mathrm{det}D)^{\frac{1}{2}}+1|=(\mathrm{det}D)^{\frac{1}{2}}|\mathrm{det}D-1|
$$
\begin{equation}\label{d57}
\leq\|D^{-1}\|^{\frac{k}{2}}|\mathrm{det}D-1|\leq \overline{C}_1(k)|\mathrm{det}D-1|.
\end{equation}
Separately consider $|\mathrm{det}D-1|$, by definition of the determinant of the matrix
\begin{equation}\label{d58}
|\mathrm{det}D-1|=\Big|1-\sum_{j_1 j_2..j_n}(-1)^{N(j_1 j_2..j_n)}d_{1j_1}d_{2j_2}..d_{nj_n}\Big|,
\end{equation}
where $N(j_1 j_2..j_n)$ number of permutations. The first term can be presented in the following form
\begin{equation}\label{d59}
|1-d_{11}d_{22}..d_{kk}|=\Big|\sum_{p=1}^kd_{11}d_{22}d_{p-1p-1}(d_{pp}-1)\Big|,
\end{equation}
therefore, by (\ref{d58}) and (\ref{d59}) it follows that 
\begin{equation}\label{d60}
|\mathrm{det}D-1|\leq \overline{c}_8(k)\rho_s.
\end{equation}
The results (\ref{d54}), (\ref{d55}) and (\ref{d60}) obtained give an estimate for function $p(x)$
$$|p(x)|\leq  \overline{c}_9(s,k)\rho_s(1+\|x\|^2)\exp\Big(\frac{\|x\|^2}{3}\Big),$$
and
$$
\max_{\|x-1\|\leq1}|p(x)|\leq \overline{c}_9(s,k)\rho_s(2+\|x\|^2)\exp\Big(\frac{(\|x\|+1)^2}{3}\Big)
$$
$$
\leq \overline{c}_{10}(s,k)\rho_s(1+\|x\|^2)\exp\Big(\frac{\|x\|^2}{3}+\|x\|\Big).
$$
By Cauchy's estimate Lemma 9.2 \cite{B}, we now can consider derivatives of function $p(x)$
$$
|D^{\nu-\alpha}p(x)|\leq \overline{c}_{11}(s,k,\nu-\alpha)\rho_s(1+\|x\|^2)\exp\Big(\frac{\|x\|^2}{3}+\|x\|\Big).
$$
Now we obtain
$$
D^{\nu}(\phi(x)-\phi_{\overline{0},D})(x)= \Big| \sum_{0\leq\alpha\leq\nu}(D^{\nu}\phi)(x)(D^{\alpha-\nu}p)(x) \Big|
$$
$$\leq \overline{c}_{12}(s,k)\rho_s(1+\|x\|^{\|\nu\|+2})\exp\Big(-\frac{\|x\|^2}{6}+\|x\|\Big).$$
and observe that
$$
\Big|\tilde{\kappa}_{\nu_1}\dots\tilde{\kappa}_{\nu_m} D^{ {\nu_1}+\dots+{\nu_m} }\phi_{\overline{0},D}(x)
- \kappa_{\nu_1}\dots\kappa_{\nu_m} D^{{\nu_1}+\dots+{\nu_m}}\phi(x)\Big|
$$
$$\leq\Big|\tilde{\kappa}_{\nu_1}\dots\tilde{\kappa}_{\nu_m} D^{ {\nu_1}+\dots+{\nu_m} }(\phi(x)-\phi_{\overline{0},D}(x))\Big|
$$
\begin{equation}\label{d61}
+\Big|(\tilde{\kappa}_{\nu_1}\dots\tilde{\kappa}_{\nu_m} - \kappa_{\nu_1}\dots\kappa_{\nu_m}) D^{ {\nu_1}+\dots+{\nu_m} }\phi(x)\Big|.
\end{equation}
By Lemma \ref{lemma6} inequality (\ref{d46}), we have the following bound for difference of products of cumulants

\begin{equation}\label{d62}
|\tilde{\kappa}_{\nu_1}\dots\tilde{\kappa}_{\nu_m} - \kappa_{\nu_1}\dots\kappa_{\nu_m}|\leq c_2(s,k)\rho_s
\end{equation}
and also by inequality (\ref{d51}), it follows similar result for product of cumulants of truncated random vectors

\begin{equation}\label{d63}
|\tilde{\kappa}_\nu\dots\tilde{\kappa}_\nu|\leq \overline{c}_7(s,k)\rho_s.
\end{equation}
Finally, using (\ref{d61}), (\ref{d62}) and (\ref{d63}), we conclude
$$
| \kappa_{\nu_1}\dots\kappa_{\nu_m}D^{ {\nu_1}+\dots+{\nu_m} }\phi(x)- \tilde{\kappa}_\nu\dots\tilde{\kappa}_\nu  D^{{\nu_1}+\dots+{\nu_m}}\phi_{\overline{0},D}(x)|
$$
\begin{equation}\label{d64}
\leq  \overline{c}_{13}(r,k,s) \rho_{s}(1 +\|t\|^{3r+2})\exp\Big(-\frac{\|t\|^2}{6}+\|t\|\Big).
\end{equation}
The inequality (\ref{d52}) follows from (\ref{d64}) and the expression 14.74(see Lemma 7.2 \cite{B}).

To proof the second inequality, we need the following estimate for the norm of mean 
$$ \| \mathbb {E} \theta_jY_ {j} \| ^ 2 = \sum_ {i = 1} ^ k \Big (\mathbb {E} \theta_j (Y_ {ji} -X_ {ji}) \Big) ^ 2 \leq
 \sum_{i=1}^k\Big(\mathbb{E}\|\theta_jX_{j}\|\mathbb{I}(\| \theta_jX_{j}\|> 1)\Big)^2 $$
\begin{equation}\label{d65}
\leq k \Big (\mathbb {E} \| \theta_jX_ {j} \| ^ s \Big) ^ 2 \leq k \rho_s ^ 2 \leq \frac{k}{(8k)^2}=\frac{1}{64k}.
\end{equation}
It follows from (\ref{d65}) that the norm of weighted mean $A_n = \sum_ {j = 1} ^ {n} \theta_j \mathbb {E} Y_j$ is bounded 
\begin{equation}\label{d66}
\|A_n\|\leq\sum_{i=1}^n\|E\theta_iY_i\|\leq \sqrt{k}\sum_{i=1}^n\theta_i^4\delta_i^4<\frac{1}{8\sqrt{k}}.
\end{equation}
From the expression 14.74(see Lemma 7.2 \cite{B}) and the inequality (\ref{d63}) (which also holds if the primes are deleted), we get
$$
\Big|P_r (-\phi, \{\kappa_\nu\} )(x+A_n) - P_r(-\phi, \{\kappa_\nu\})(x)\Big|
$$
\begin{equation}\label{d67}
\leq \overline{c}_{14}(s,k)\max_{0\leq |\nu|\leq 3r}\Big(\Big|D^\nu\phi(x+A_n)-D^\nu\phi(x)\Big|\Big).
\end{equation}
Denote $D^\nu\phi(x)=q_\nu\phi$, where $q_\nu\phi$ is a polynomial of degree $|\nu|$ (with coefficients depending only on vector $\nu$ and dimension  $k$), so that
$$|D^\nu\phi(x+A_n)-D^\nu\phi(x)|=|q_\nu(x+A_n)-q_\nu(x)|\phi(x)
$$
\begin{equation}\label{d68}
+|
q_\nu(x+A_n)||\phi(x+A_n)-\phi(x)|.
\end{equation}
But, using the inequalities (\ref{d66}) and  (\ref{d41}), we also can now get the following 
$$
|(x+A_n)^\alpha-x^\alpha|\leq|\alpha|(\|x+A_n\|^{|a|-1}+\|x\|^{|a|-1})\|A_n\|
$$
\begin{equation}\label{d69}
\leq \overline{c}_{15}(s,k)\rho_s(1+\|x\|^{|\alpha|-1})
\end{equation}
and 
$$
|(x+A_n)^\alpha(\phi(x+A_n)-\phi(x))|=|(x+A_n)^\alpha\phi(x)|\Big|\exp\Big(-\frac{1}{2}\|x+A_n\|^2+\frac{1}{2}\|x\|^2\Big)-1\Big|
$$
$$
\leq \Big|\frac{1}{2}(x+A_n)^\alpha\phi(x)|(\|x+A_n\|^2-\|x\|^2)\exp\Big(-\frac{1}{2}\|x+A_n\|^2+\frac{1}{2}\|x\|^2\Big)\Big|
$$
$$
\leq |(x+A_n)|^\alpha\phi(x)\|A_n\|(1+\|x\|)\exp\Big(\|A_n\|(1+\|x\|)\Big)\Big|
$$
\begin{equation}\label{d70}
\leq \overline{c}_{16}(\alpha,k)\rho_s(1+\|x\|^{\|\alpha\|+1})\exp\Big(-\frac{\|x\|^2}{2}+\frac{\|x\|}{8k^{1/2}}\Big).
\end{equation}
Relations (\ref{d69}), (\ref{d70}) are used in (\ref{d68}) to yield
\begin{equation}\label{d71}
|D^\nu\phi(x+A_n)-D^\nu\phi(x)|  \leq \overline{c}_{17}(\nu,k)\rho_s(1+\|x\|^{\|\nu\|+1})\exp\Big(-\frac{\|x\|^2}{2}+\frac{\|x\|}{8k^{1/2}}\Big).
\end{equation}
Finally, (\ref{d71}) is used in (\ref{d67}) to get (\ref{d53}).
\end {proof}

\begin {lemma} \label {lemma6}
Let $ \rho_s  \leq (8k)^{-1}, $
then for $ \| t \| \leq (16\rho_3)^{-1} $ and for any nonnegative integer vector $\alpha$,  one has
\begin {equation} \label {d72_0}  
 \Big | D ^ \alpha \prod_ {j = 1} ^ n \varphi_j (\theta_j t) \Big | \leq c_5 (\alpha, k) (1+ \| t \| ^ {| \alpha |}) \exp \Big (- \frac {5} {48} \| t \| ^ 2 \Big), 
 \end{equation}
where $ \rho_s  $  and $ \varphi_j (t) $ are defined in (\ref {d6}) and (\ref{d17_2}), respectively.
\end {lemma}
\begin {proof}
The proof of Lemma follows the scheme of proofs of  Lemma 14.3 \cite {B}.
For the truncated random vectors $ \theta_jZ_j = \theta_j(Y_j- \mathbb {E} Y_j)$, where $Y_j = X_j \mathbb{I} (\| \theta_j X_j \| \leq 1),\;j=1\dots n$, we use similar to (8.46 from \cite {B}) inequality for characteristic function 
$$
| \varphi_j (\theta_j t) |^2 \leq 1- \theta_j ^ 2 \mathbb {E} \langle Z_j, t \rangle ^ 2 + \frac {2} {3} | \theta_j | ^ 3 \mathbb {E} | \langle Z_j, t \rangle | ^ 3 $$
$$ \leq \exp \Big (- \theta_j ^ 2 \mathbb {E} \langle Z_j, t \rangle ^ 2 + \frac {2} {3} | \theta_j | ^ 3 \mathbb {E} | \langle Z_j, t \rangle | ^ 3 \Big), $$
further note that
$$
\exp \Big (\theta_j ^ 2 \mathbb {E} \langle Z_j, t \rangle ^ 2- \frac {2} {3} | \theta_j | ^ 3 \mathbb {E} | \langle Z_j, t \rangle | ^ 3 \Big)
$$
$$
\leq \exp \Big ((\mathbb {E} | \theta_j \langle Z_j, t \rangle | ^ 3) ^ {\frac {2} {3}} - \frac {2} {3} | \theta_j | ^ 3 \mathbb {E} | \langle Z_j, t \rangle | ^ 3 \Big) \leq \exp \Big (\frac {1} {3} \Big).
$$
Now we denote the subset $ N_r = \{j_1, j_2 ,\dots, j_r \} $ as a subset of $ N = \{1,2,\dots,n \} $, consisting of $ r $ elements, for $ \| t \| \leq (16\rho_3)^{-1} $ the following chain of inequalities holds
$$
\Big | \prod_ {j \in N \setminus N_r} \varphi_j (\theta_jt) \Big |^2 
\leq \exp \Big (\sum_ {j \in \mathbb {N} \setminus N_r} \Big [-  \theta_j ^ 2 \mathbb {E} \langle Z_j, t \rangle ^ 2 +\frac {2} {3} | \theta_j | ^ 3 \mathbb {E} | \langle Z_j, t \rangle | ^ 3 \Big] \Big)
$$
$$
\leq \exp \Big (- \langle Dt, t \rangle  + \frac {1} {3} \| t \| ^ 2 \Big) \exp \Big (\frac {r} {3} \Big) \leq \exp \Big (\Big (- \frac {3} {4} + \frac {1} {3}  \Big) \| t \| ^ 2 \Big) \exp \Big (\frac {r} {3} \Big)
$$
\begin {equation} \label {d31}
 = \exp \Big (- \frac {5} {12} \| t \| ^ 2 \Big) \exp \Big (\frac{r} {3} \Big).
\nonumber
\end {equation}
For $ \alpha = 0 $ the statement of the Lemma is proved.

Before proceeding to the proof of the case $ \alpha \neq 0 $, consider the modulus of the derivative of the characteristic function of the random vector $ Z_j $
\begin {equation} \label {d32}
   \Big | D_m \varphi_j (\theta_j t) \Big | = | \theta_j | \Big | \mathbb {E} Z_ {j, m} \exp \Big (i \langle \theta_jt, Z_j \rangle \Big) \Big |.
\end {equation}
If a positive vector $ \beta $ satisfies the condition $ | \beta | = 1 $ then, using (\ref{d32}), we have 
$$
\Big | D ^ {\beta} \varphi_j (\theta_j t) \Big | = | \theta_j | \Big | \mathbb {E} Z_ {j, \beta} \Big (\exp \Big (i \langle \theta_jt, Z_j \rangle \Big) -1 \Big) \Big |
$$

\begin {equation} \label {d33} 
\leq | \theta_j | \mathbb {E} \Big | Z_ {j, \beta} \langle \theta_jt, Z_j \rangle \Big |
\leq \theta_j ^ 2 \| t \| \mathbb {E} \| Z_j \| ^ 2 \leq \theta_j ^ 2 \tilde{\rho}_{2,j} \| t \|, 
\end {equation}
we also, still using (\ref{d32}), conclude that for any positive vector $ \beta $ with $ | \beta | \geq 2 $  

\begin {equation} \label {d34}
\Big | D ^ \beta \varphi_j (\theta_j t) \Big | \leq | \theta_j | ^ {| \beta |} \mathbb {E} | Z_j ^ \beta | \leq 2 ^ {| \beta |} \theta_j ^ 2 \rho_{2,j}. 
\end {equation}
Finally, collecting the results (\ref{d33}) and (\ref{d34}), we get that for any non-negative vector $ \beta> 0 $, one has
\begin {equation} \label {d35}
     \Big | D ^ \beta \varphi_j (\theta_j t) \Big | \leq \overline{c}_{18} (\alpha, k) \theta_j ^ 2 \rho_{2,j} \max \{1, \| t \| \}.\nonumber
\end {equation}
Now consider the positive vector $ \alpha> 0 $, according to the rule of differentiation of the product of functions, we obtain that
\begin {equation} \label {d36}
    D ^ \alpha \prod_ {j = 1} ^ n \varphi_j (\theta_j t) = \sum \prod_ {j \in N \setminus N_r} \varphi_j (\theta_j t) D ^ {\beta_1} \varphi_ { j_1} (\theta_ {j_1} t) \dots D ^ {\beta_r} \varphi_ {j_r} (\theta_ {j_r} t),
\end {equation}
where $ N_r = \{j_1,\dots,j_r \}, \; 1 \leq r \leq | \alpha |, $
$\beta_1, \beta_2,\dots, \beta_r $ are vectors that satisfy the conditions
$ | \beta_j | \geq 1 \; \; (1 \leq j \leq r) $ and $ \sum_ {j = 1} ^ r \beta_j = \alpha. $
The number of multiplications in each of the $ n ^ {|\alpha|} $ terms of the expression (\ref {d36}) is
$$ \frac {\alpha_1! \dots \alpha_k!} {\prod_ {j = 1} ^ r \prod_ {i = 1} ^ k \beta_ {ji}! }, $$
where $ \alpha = (\alpha_1, \dots, \alpha_k) $ and $ \beta_j = (\beta_ {j_1}, \dots, \beta_ {j_k}), \; 1 \leq j \leq r. $ Each term in the expression (\ref {d36}) is bounded by the value
\begin {equation} \label {d37}
\exp \Big (\frac {r} {6} - \frac {5} {48} \| t \| \Big) \prod_ {j \in N_r} b_j, 
\end{equation}
where $ b_j = \overline{c}_{18} (\alpha, k) \rho_{2,j} \theta_j ^ 2 \max \{1, \| t \| \}, $
therefore from (\ref {d36}), (\ref {d37}) we obtain
\begin {equation} \label {d37_2} \Big | D ^ \alpha \prod_ {j = 1} ^ n \varphi_j (\theta_j t) \Big | \leq \sum_ {1 \leq r \leq | \alpha |} \overline{c}_{19} (\alpha, r) \exp \Big (\frac {r} {6} - \frac {5} {48} \| t \| \Big) \sum_r \prod_ {j \in N_r } b_j, 
\end{equation}
where the outer summation is over all $ r $ elements from $ N. $
It remains to evaluate the expression
\begin {equation} \label {d37_3}  
\sum_r \prod_ {j \in N_r} b_j \leq \Big (\sum_ {j = 1} ^ n b_j \Big) ^ r = (\overline{c}_{18} (\alpha, k) \rho_2 \max \{1, \| t \| \}) ^ r  \leq
(\overline{c}_{18} (\alpha, k) k) ^ r (1+ \| t \| ^ r),
\end{equation}
putting (\ref{d37_3}) into (\ref{d37_2}), we get (\ref{d72_0} ).
  \end {proof}

\begin {lemma} \label {lemma7}
Let $s=5$, then  for any nonnegative integer vector $\alpha$, one has
\begin{equation}\label{d72}
\mathbb {E_\theta}\int \limits _ {\frac { \lambda\sqrt {n}} {\beta_3} \leq \| t \| \leq \frac {n^{3/2}} {\beta _ 5}} \Big | D ^ \alpha \prod_ {j = 1} ^ n \varphi_j (\theta_jt) \Big | dt \leq c_6 (\alpha,\lambda, k) \frac {\beta_5 } {n^{3/2}},
\end{equation}
where $ \beta_5 $   and $ \varphi_j (t) $ are defined in  (\ref {d6}) and (\ref{d17_2}) 
 , respectively, $\lambda$ is any positive constant.
\end {lemma}
\begin {proof}
The proof of Lemma follows the scheme of the proof of Lemma 3.5 \cite {KS}.
Let us estimate the modulus of the characteristic function of the truncated random vector $Y_j = X_j\mathbb{I} (\| \theta_jX_j \| \leq 1),
$
(see (\ref{d8}), (\ref{d8})), denote $ \xi_j(t) $ as the characteristic function of the random vector $ X_j. $ First, using the Chebyshev inequality, we obtain the estimate
\begin {equation} \label {d74}
\mathbb {E} \mathbb{I} (\| \theta_jX_j \| \geq 1) = P (\| \theta_jX_j \| \geq 1)  \leq \mathbb {E} \| X_j \| ^ 2 \theta_j ^ 2 = k \theta_j ^ 2,
\end{equation}
then applying some transformations and using (\ref{d74}), we obtain the following chain of inequalities 
$$
\Big | \mathbb {E} \exp \Big (i \theta_j \langle t, Y_j \rangle \Big)  \Big | = \Big | \mathbb {E} \exp \Big (i \theta_j \langle t, X_j \rangle \mathbb{I} (\| \theta_jX_j \| \leq 1) \Big) \Big|
$$
$$
= \Big | \mathbb {E} \exp \Big (i \theta_j \langle t, X_j \rangle \mathbb{I} (\| \theta_jX_j \| \leq 1) \Big) \Big (\mathbb{I} (\| \theta_jX_j \| \leq 1) + \mathbb{I} (\| \theta_jX_j \|> 1) \Big) \Big|
$$
$$
 = \Big | \mathbb {E} \exp \Big (i \theta_j \langle t, X_j\rangle \Big) \mathbb{I} (\| \theta_jX_j \| \leq 1) +\mathbb {E} \mathbb{I} (\| \theta_jX_j \|> 1) \Big|
$$
$$
= \Big | \mathbb {E} \Big [\exp \Big (i \theta_j \langle t, X_j \rangle \Big) - \exp \Big (i \theta_j \langle t, X_j \rangle \Big) \mathbb{I} (\| \theta_jX_j \|> 1) \Big]
+ \mathbb {E} \mathbb{I} (\| \theta_jX_j \|> 1) \Big | 
$$
$$
 = \Big | \mathbb {E} \exp \Big (i \theta_j \langle t, X_j \rangle \Big)  + \mathbb {E} \mathbb{I} (\| \theta_jX_j \|> 1) \Big (1 - \exp \Big (i \theta_j \langle t, X_j \rangle \Big)\Big|
$$
$$
\leq | \xi_j (\theta_jt) | + \Big | \mathbb {E} \mathbb{I} (\| \theta_jX_j \|> 1) \Big (1- \exp \Big (i \theta_j \langle t, X_j \rangle \Big) \Big|
$$
$$
\leq | \xi_j (\theta_jt) | + \mathbb {E} \mathbb{I} (\| \theta_jX_j \|> 1) \Big | 1- \exp \Big (i \theta_j \langle t, X_j \rangle\Big) \Big | 
$$
\begin {equation} \label {d74_1}
\leq | \xi_j (\theta_jt) | +2 \mathbb {E} \mathbb{I} (\| \theta_jX_j \|> 1) \leq | \xi_j (\theta_jt) | + 2k \theta_j ^ 2.
\end {equation}
next we are going to show that for any $ r $ the following inequality holds
\begin {equation} \label {d75}
\mathbb {E}_\theta | \xi_j (\theta_jr) | ^ 2 \leq 1 - \hat{c}_{1} \min \Big \{\frac {\| r \| ^ 2} {n},\frac{1}{ \rho_{4,j}}  \Big \}. 
\end {equation}
For $r=0$ the inequality holds automatically, therefore the case with $r>0$ is considered below. Let us denote $ X_j '$ as an independent copy of the random vector $ X_j $,
and define a random vector $ \hat {X} = X_j-X'_j$. If we denote  $ J_n $ as characteristic function of the component of a random vector uniformly distributed on the unit sphere, $ \theta_j. $ Then 
$$
\mathbb {E}_\theta | \xi_j (\theta_jr) | ^ 2 =\mathbb {E}_\theta\mathbb {E}(\exp(i\theta_j\langle r,\hat {X}  \rangle))=\mathbb {E}\mathbb {E}_\theta(\exp(i\theta_j\langle r,\hat {X} \rangle))= \mathbb {E} J_n (\langle r, \hat {X} \rangle)
$$
\begin {equation} \label {d76}
= \mathbb {E} J_n \Big (\| r \| \frac {\langle r, \hat {X} \rangle} {\| r \|} \Big).
\end {equation} 
Lemma 3.3 \cite {KS} implies that the estimate holds for the characteristic function of a random variable uniformly distributed on the unit sphere
\begin {equation} \label {d77}
\mathbb {E} J_n \Big (\| r \| \frac {\langle r, \hat {X} \rangle} {\| r \|} \Big) \leq 1- \hat {c} _{2} \mathbb {E} \min \Big \{\frac {\| r \| ^ 2} {n} \Big (\frac {\langle r, \hat {X} \rangle} {\| r \|} \Big) ^ 2,1 \Big \}.
\end {equation} 
We have to define the random variable $ X '' $ as
$ X '' = \frac {\langle r, \hat {X} \rangle ^ 2} {2 \| r \| ^ 2} $ and $ \tau = \frac {\| r \| ^ 2} {n}, $
then
\begin {equation} \label {d78}
\mathbb {E} X '' = \mathbb {E} \Bigg (\frac {\sum \limits_ {i = 1} ^ {k} r_i ^ 2 \hat {X} _i ^ 2 + 2 \sum \limits_ {1 \leq j <i \leq k} r_jr_i \hat {X} _j \hat {X} _i} {2 \| r \| ^ 2} \Bigg) = \frac {2 \| r \| ^ 2} {2 \| r \| ^ 2} = 1, 
\end {equation} 

\begin {equation} \label {d79}
\mathbb {E} (X '') ^ 2 = \frac {\langle r, \hat {X} \rangle ^ 4} {4 \| r \| ^ 4} \leq \mathbb {E} \frac {\| r \| ^ 4} {4 \| r \| ^ 4} \| \hat {X} \| ^ 4 \leq \frac {2 \rho_{4,j} + 6k ^ 2 } {4} \leq2 \rho_{4,j}.
\end {equation} 
Let us show that $ \mathbb {E} \min \{\tau X '', 1 \} \geq \hat {c}_3 \min \{\tau, 1/\rho_{4,j}\}. $
Since the right-hand side increases with $ \tau $, it suffices to prove the inequality for $ \tau <(10 \rho_{4,j})^{-1}. $ Using (\ref{d78}), (\ref{d79}) and  Lemma 3.1 \cite {KS}, we have
$\mathbb {E} \mathbb{I} (X '' \leq 10 \rho_{4,j}) X '' \geq 4/5,$
so for $ 0 <\tau \leq (10 \rho_{4,j})^{-1} $ one has
$$
\mathbb {E} \min \{\tau X '', 1 \} \geq \mathbb {E} \mathbb{I} (X '' \leq 10 \rho_{4,j}) \min \{\tau X '', 1 \} 
$$
\begin {equation} \label {d80}
= \tau \mathbb {E} \mathbb{I} (X '' \leq 10 \rho_{4,j}) X ''> \frac {\tau} {2},
\end {equation} 
Finally, using (\ref{d76}), (\ref{d77}), (\ref{d80}), we get (\ref{d75}).

Similarly, consider $ Z_j-Z'_j $, where $ Z'_j $ is an independent copy of $ Z_j $, also note that $ \mathbb {E} \theta_j ^ 2 = n^{-1}$ 
and $\mathbb {E} \theta_j ^ 4 =3/n(n+2).$ (see Lemma 5.3.1 \cite{Bob} ).
And, using the inequalities (\ref {d74_1}) and (\ref {d75}), we obtain the estimate
$$ \mathbb {E}_\theta | \varphi_j (\theta_jt) | ^ 2 = \mathbb {E}_\theta \Big [\mathbb {E}  \exp \Big (i \theta_j \langle t, Z_j-Z'_j \rangle \Big)  \Big]
= \mathbb {E}  \Big [\mathbb {E}_\theta  \exp \Big (i \theta_j \langle t, Y_j-Y'_j \rangle \Big) \Big] 
$$
$$
\leq \mathbb {E}_\theta  \Big [\Big (| \xi_j (\theta_jt) | + 2k \theta_j ^ 2 \Big) \Big (| \xi_j (- \theta_jt) | + 2k \theta_j ^ 2 \Big ) \Big] 
\leq \mathbb {E}_\theta  | \xi_j (\theta_j t) | ^ 2 + \frac {2k} {n} + \frac { 12k ^ 2} {n(n+2)}
$$
\begin {equation} \label {d81}
\leq 1-\hat{c}_1 \min \Big \{\frac {\| t \| ^ 2} {n},\frac{1}{ \rho_{4,j}} \Big \} +  \frac {2k} {n} +\frac {12k ^ 2} {n^2}.
\end {equation}
Now since $ \rho_{4,j} \geq k ^ 2,\;j=1\dots n,$ and $ k^3\leq(\beta_3)^2\leq k\beta_4 $, we can note that 
 \begin {equation} \label {d81_2}
 \min \Big \{\frac {\lambda^2} {\beta_3^2},\frac{1}{ \rho_{4,j}} \Big \}\leq
\min \Big \{\frac {\lambda^2} {k^3},\frac{1}{ k^2} \Big \}\leq\frac{1}{k^2}
\end {equation}
and therefore in the considered region $ \| t \| ^ 2 \geq \lambda^2n/ \beta_3^2$  the following inequality holds
\begin {equation} \label {d82}
\Big (1-\hat{c}_1 \min \Big \{\frac {\| t \| ^ 2} {n},\frac{1}{ \rho_{4,j}} \Big \} +  \frac {2k} {n} + \frac {12k} {n^2} \Big) ^ {- 1} \leq \Big(1- \frac {\hat{c}_1} {k ^ 2}\Big) ^ {- 1}. 
\end {equation} 
In Lemma \ref {lemma4} it was shown (\ref{d35}) that for any positive integer vector $\alpha$, one has
$$ \Big | D ^ \alpha \varphi_j (\theta_j t) \Big | \leq \overline{c}_{18} (\alpha, k) \theta_j ^ 2 \rho_{2,j} \max \{1, \| t \| \},
$$
from this it immediately follows
$$ \Big (\mathbb {E}_\theta \Big | D ^ \alpha \varphi_j (\theta_j t) \Big | ^ 2 \Big) ^ {\frac {1} {2}} \leq \overline{c}_{18} (\alpha, k) \rho_{2,j} \max \{1, \| t \| \} \Big (\mathbb {E}_\theta \theta_j ^ 4 \Big) ^ {\frac {1} {2}} 
  $$
 \begin {equation} \label {d83}
 \leq \overline{c}_{18} (\alpha, k) \rho_{2,j} \max \{1, \| t \| \} \frac { \sqrt{3}} {n}.
\end {equation}
Further, using Theorem 1 \cite {S}, we obtain
$$
\mathbb {E}_\theta  \Big | \Big [\prod_ {j \in N_r} \varphi_j (\theta_jt) \Big] D ^ {\beta_1} \varphi_ {j_1} (\theta_ {j_1}) \dots D ^ {\beta_r} \varphi_ { j_r} (\theta_ {j_r}) \Big |
$$
\begin {equation} \label {d84}
\leq \Big [\prod_ {j \in N_r} \left (\mathbb {E}_\theta  | \varphi_j (\theta_jt) | ^ 2 \right) ^ {\frac {1} {2}} \Big] \left (\mathbb {E}_\theta  | D ^ {\beta_1} \varphi_ {j_1} (\theta_ {j_1}) | ^ 2 \right) ^ {\frac {1} {2}} \dots \left (\mathbb {E}_\theta  | D ^ {\beta_r} \varphi_ {j_r} (\theta_ {j_r}) | ^ 2 \right) ^ {\frac {1} {2}}.
\end {equation}
Let us denote the subset of indices $\mathfrak {G} = \Big\{j: \rho_{4,j} <2 \beta_4\Big \} 
$
and we come to the conclusion that
$$ 
\beta_4 = \frac {1} {n} \sum \limits_ {j = 1} ^ {n} \rho_{4,j} \geq \frac {1} {n} \sum \limits_ {j \notin \mathfrak {G}} \rho_{4,j} \geq \frac {n- | \mathfrak {G} |} {n} 16 \beta_4, 
$$
therefore, due to $ | \mathfrak {G} | \geq  n/2, $ $1+x\leq \exp(x)$ and (\ref{d81}) the following chain of inequalities is valid 
$$ 
%\mathbb {E}_\theta  \Big | \prod \limits_ {j = 1} ^ {n} %\varphi_j (\theta_jt) \Big |\leq 
\prod \limits_ {j = 1} ^ {n} \Big(\mathbb {E}_\theta  | \varphi_j (\theta_jt) |^2\Big)^{\frac{1}{2}}\leq \prod \limits_ {j \in \mathfrak {G}} \Big(\mathbb {E}_\theta | \varphi_j (\theta_jt)  |^2\Big)^{\frac{1}{2}} 
$$
$$
\leq \prod \limits_ {j \in \mathfrak {G}} \Big (1-\hat{c}_4 \min \Big \{\frac {\| t \| ^ 2} {n},\frac{1}{ \rho_{4,j}} \Big \} + 2k \frac {1} {n} + 12 \frac {k ^ 2} {n ^ 2} \Big) ^ {\frac {1} {2}}
$$
$$
\leq \Big (1- \overline {c} _5 \min \Big \{\frac {\| t \| ^ 2} {n}, \frac{1}{ \beta_4}\Big \} + 2k \frac {1} {n} + 12 \frac {k ^ 2} {n ^ 2} \Big) ^ {\frac {n} {4}}
$$
\begin {equation} \label {d85}
\leq \exp \Big (\frac {k} {2} + 3k ^ 2 \Big) \exp \Big (- \hat{c}_5 \min \Big \{\| t \| ^ 2, \frac {n} {\beta_4} \Big \} \Big). 
\end {equation} 
next by (\ref{d82}) and (\ref{d85}) we come to the conclusion
\begin {equation} \label {d86}
\prod \limits_ {j \in N_r} \Big (\mathbb {E}_\theta | \varphi_j (\theta_jt) | ^ 2 \Big) ^ {\frac {1} {2}} \leq \exp \Big (\frac {k} {2} + 3k ^ 2 \Big) \Big(1- \frac {\hat{c}_1} {k ^ 2}\Big) ^ {- r} \exp \Big (- \hat {c} _5 \min \Big \{\| t \| ^ 2, \frac {n} {\beta_4} \Big \} \Big), 
\end {equation}
and similarly, as in the proof of Lemma \ref {lemma4}, using (\ref{d83}), (\ref{d84}), (\ref{d86}) we obtain the estimate
\begin {equation} \label {d87}
\mathbb {E}_\theta  \Big | D ^ \alpha \prod_j ^ n \varphi_j (\theta_j t) \Big | \leq \hat{c}_6 (\alpha, k) (1+ \| t \| ^ {| \alpha |}) \exp \Big (- \hat {c} _5 \min \Big \{\| t \| ^ 2 , \frac {n} {\beta_4} \Big \} \Big)
\end {equation}
and therefore, using (\ref{d81_2}), $\Big(\beta_4/k^2\Big)^{1/2}\leq \Big(\beta_5/k^{5/2}\Big)^{1/3}$ (see Lemma 1 \cite{Ay}) and also noting that $  \min \Big \{\frac {\lambda^2} {\beta_3^2},\frac{1}{ \beta_4} \Big \}\geq \min \Big \{\frac{\lambda^2} {k\beta_4},\frac{1}{ \beta_{4}} \Big \}\geq \beta_4^{-1}\hat{c}_7(k,\lambda)$, we get
$$
\int \limits _ {\frac {\lambda \sqrt{n}} {\beta_3} \leq \| t \| \leq \frac {n} {\beta_4}} \mathbb {E}_\theta  \Big | D ^ \alpha \prod_ {j = 1} ^ {n} \varphi_j (\theta_jt) \Big | dt
$$
$$ \leq
\int \limits _ {\frac {\lambda \sqrt{n}} {\beta_3} \leq \| t \| \leq \frac {n} {\beta_4}} \hat{c}_6 (\alpha, k) (1+ \| t \| ^ {| \alpha |}) \exp \Big (- \hat {c} _5 \min \Big \{\| t \| ^ 2, \frac {n} {\beta_4} \Big \} \Big) dt
$$
$$
\leq \int \limits _ {\frac {\lambda \sqrt{n}} {\beta_3}\leq \| t \|
\leq \frac {n} {\beta_4}} \hat{c}_6  (\alpha, k) (1+ \| t \| ^ {| \alpha |}) \exp \Big (- \hat {c} _7 (k,\lambda)\frac {n} {\beta_4} \Big) dt
$$
$$ \leq \hat{c}_8 (\alpha, k, \lambda) \exp \Big (- \hat {c}_9 (k,\lambda) \frac {n} {\beta_4} \Big)\leq \hat{c}_{10}(\alpha, k,\lambda)\Big(\frac{\beta_4}{n}\Big)^{3/2} \leq c_6(\alpha, k,\lambda)\frac{\beta_5}{n^{3/2}} 
.$$
\end {proof}

\section {Proof of the main theorem} \label {3}
\begin {proof}
We use the following notation for the distributions of random vectors : over all Borel sets 
$$
F_ {X} (B) = P \Big (\sum_ {j = 1} ^ n \theta_jX_j \in B \Big),\;\;
F_ {Y} (B) = P \Big (\sum_ {j = 1} ^ n \theta_jY_j \in B \Big),$$
$$
F_ {Z} (B) = P \Big (\sum_ {j = 1} ^ n \theta_jZ_j \in B \Big).
$$
At first, assume that $ \rho_{5} \leq (8k)^{-1}$ and split the original integral into several terms, for this we add and subtract the distribution of the sum of truncated random variables (\ref {d8}) and  the Edgeworth expansion (\ref{d25_1}) of fourth order (we can get the explicit view of signed measures $\Psi_4$ by Lemma \ref{lemma1} and (\ref{d19}) ).
\begin {equation} \label {d88}
I=\Big | \int \limits_ {B} d (F_X-G \Big | \leq \Big | \int \limits_ {B} d (F_X-F_Y) \Big | + \Big | \int \limits_ {B} d(\Psi_4-G)\Big |+ \Big | \int \limits_ {B} d (F_Y-\Psi_4)\Big |=I_1+I_2+I_3.
\end {equation} 
At first, we consider the first integral $I_1$ and use the following transformation
$$ I_1 \leq \Big | \int \limits_ {B} d \Big (\sum_ {j = 1} ^ nF_ {X_j} \ast\dots\ast F_ {X_ {j-1}} \ast (F_ {X_ {j}} - F_ {Y_ {j}}) \ast F_ {Y_ {j + 1}} \ast\dots\ast F_ {Y_ {n}} \Big) \Big | $$
\begin {equation} \label {d89}
\leq \sum_ {j = 1} ^ n \Big | \int \limits_ {B} d (F_ {X_j} -F_ {Y_j}) \Big | \leq \sum_ {j = 1} ^ nP (\| \theta_jX_j \|> 1) \leq \rho_5,
\end {equation} 
where $ "\ast" $ denotes the convolution of two functions and distributions are defined as $F_ {X_j} (B) = P (\theta_jX_j \in B)$, 
$F_ {Y_j} (B) = P (\theta_jY_j \in B), \; j = 1,\dots,n.$

Next we estimate the second integral $I_2$ using the same scheme as Bobkov in \cite{Bob}. Since  for identically distributed random vectors $|\kappa_\nu|=l_4(\theta)|\kappa_{\nu,1}|$, where $ l_p(\theta)=\sum_{j=1}^n|\theta_j|^p$ , we get one
$$
I_2
=\Big | \int \limits_ {B} \frac{3}{n} P_2(-\phi:\{ \kappa_{\nu,1} \})(x)- l_4(\theta)P_2(-\phi:\{ \kappa_{\nu,1} \})(x)dx\Big |
$$
\begin {equation} \label {d90}
\leq \int \limits_ {B}\Big(\Big|\frac{3}{n+2}-l_4(\theta)\Big|+\Big|\frac{6}{n(n+2)}\Big|\Big)\Big|P_2(-\phi:\{\kappa_{\nu,1}\})(x)dx\Big|,
\end {equation} 
recalling that $|\kappa_{\nu,1}|\leq c(\nu)\rho_{|\nu|,1}$, $\rho_{4,1}=\beta_4$ and $k^2\leq\beta_4\leq (\beta_5)^{ \frac{4}{5} }\leq\beta_5$ we have
\begin {equation} \label {d91}
\Big|\int P_2(-\phi:\{\kappa_{\nu,1}\})(x)dx\Big|=\Big|\int \sum_{|\nu|=4}\frac{\kappa_{\nu,1}}{\nu!}D^\nu\phi(x)dx\Big|
\leq\hat{C}_1(k)\beta_4\leq\hat{C}_1(k)\beta_5,
\end {equation} 
also (see in details chapter 5.3 \cite{Bob} ) by Lemma  5.3.1 \cite{Bob} variance of $l_4(\theta)$ is bounded 
\begin {equation} \label {d92_2}
\mathbb{E}_\theta|l_4(\theta)-\mathbb{E}_\theta l_4(\theta)|\leq \mathbb{E}_\theta|l_4(\theta)-\mathbb{E}_\theta l_4(\theta)|^2)^{\frac{1}{2}}\leq(24/n^3)^{\frac{1}{2}}.
\end {equation} 
 Summing up the results, we use the  (\ref{d90}), (\ref{d91}) and  (\ref{d92_2}) and get the estimation
\begin {equation} \label {d91_1}
I_2\leq C_1(k) \frac{\beta_5}{n^{3/2}}.
\end {equation}
To estimate the third term $I_3$, we  split this integral into the sum of the three ones
$$ 
\Big | \int \limits_ {B} d F_Y-\Psi_4)\Big | = \Big | \int \limits_ {B_n = B + \{A_n\}} d F_Z-\Psi_4^{\prime\prime}\Big| 
$$
\begin {equation} \label {d92}
\leq \Big | \int \limits_ {B_n} d (F_Z- \Psi_4^{\prime} )\Big | + \Big | \int \limits_ {B_n} d (\Psi_4^{\prime}- \Psi_4) \Big | + \Big | \int \limits_ {B_n} d (\Psi_4^{\prime\prime} - \Psi_4) \Big |=I_1^{\prime}+I_2^{\prime}+I_3^{\prime}.
\end {equation}
where $\Psi_4^{\prime\prime}(B)=\Psi_4(B+A_n).$ Using Lemma \ref{lemma7}, we can show that the integrals $I_1^{\prime}$ and $I_2^{\prime}$ are bounded 
\begin {equation} \label {d93}
I_1^{\prime}+I_2^{\prime} \leq  C_2(k)\rho_5.
\end {equation}
Now we consider $I_3^{\prime}$, for this we use the smoothing inequality (Corollary 11.5 [2])
\begin {equation} \label {d94}
I_3^\prime\leq
\hat{C}_3 \int \Big | d ((F_Z-\Psi_4^{\prime}) \ast K_ \epsilon) \Big | + C_4 (k) \epsilon,
\end {equation}
where $ \epsilon = \hat{c}(k)\beta_5/n^{3/2}$ and $ K_ \epsilon (x) $ is a kernel function (see 13.8-13.13 in \cite {B}),
whose characteristic function $ \widehat {K}_\epsilon (t) = 0 $ for $ \| t \|> n^{3/2}/\beta_5. $
By Lemma 11.6 \cite {B}, one has
$$
\int \limits \Big | d ((F_Z- \Psi_4^{\prime}) \ast K_ \epsilon) \Big | \leq \hat {C}_4 (k) \max \limits_ {0 \leq | \alpha + \beta | \leq k + 1} \int \Big | D ^ \alpha (\hat {F} _Z- \hat{\phi}_{\overline{0}, D} (t)) D ^ \beta \hat {K} _ \epsilon (t) \Big | dt.
$$
Since $| D ^ \beta \hat {K} _ \epsilon (t) | \leq \hat {c_1}, $ we get that
\begin {equation} \label {d94_1}
\int \limits \Big | D ^ \alpha (\hat {F} _Z- \hat{\phi}_{\overline{0}, D}) (t) D ^ \beta \hat {K} _ \epsilon (t) \Big | dt \leq \hat {C} _5 (k) \int \limits _ {\| t \| \leq \frac{n^{3/2}}{\rho_5}} \Big | D ^ \alpha (\hat {F} _Z- \hat{\phi}_{\overline{0}, D}) (t) \Big | dt .
\end {equation}
Now we denote the values $ E_n = c_1 (k, k + 5) \min \Big\{\eta_ {k + 5} ^ {- \frac {1} {k + 5}}, \eta_ {k + 3} ^ {- \frac {1} {k + 5}} \Big\}$ and $E_n^\prime=\sqrt{4/5}E_n $ (see (\ref {d12})), where $ c_1 (k, k + 5) $ is a constant from the statement of Lemma 3 \cite{Ay}, further, we add and subtract terms of the asymptotic expansion of the logarithm of the characteristic function. Considering that by definition $ \hat {F} _Z = \prod_ {j = 1} ^ n \varphi_1 (\theta_jt)$ (see (\ref{d17_1}), (\ref{d17_2})), we can get, similarly to (71) \cite {Saz}, the following inequality
$$ 
\int \limits _ {\| t \| \leq \frac{n^{3/2}}{\beta_5}} \Big | D ^ \alpha (\hat {F} _Z- \hat{\phi}_{\overline{0}, D}) (t) \Big | dt
$$
$$ 
\leq
\int \limits _ {\| t \| \leq E_n^\prime } \Big | D ^ \alpha \Big [\prod_ {j = 1} ^ n \varphi_1 (\theta_jt) - \exp \Big (- \frac {1} {2} \langle Dt, t \rangle \Big) \sum_ { r = 0} ^ {k+2} \hat{P}_r (it: \{ \tilde{\kappa}_\nu\}) \Big] \Big | dt
$$

$$ + \int \limits _ { E_n^\prime \leq \| t \| \leq \frac{1} {16\rho_3} } \Big | D ^ \alpha \prod_ {j = 1} ^ n \varphi_1 (\theta_jt) \Big | dt
+ \int \limits_{\frac{1} {16\rho_3}  \leq \| t \| \leq \frac{n^{3/2}}{\beta_5} } \Big | D ^ \alpha \prod_ {j = 1} ^ n \varphi_1 (\theta_jt) \Big | dt
$$

$$ 
+ \int \limits _ { E_n^\prime \leq \| t \|} \Big | D ^ \alpha \exp \Big (- \frac {1} {2} \langle Dt, t \rangle \Big) \Big | dt
+ \int \Big | D ^ \alpha \sum_ {r = 3} ^ {k+2} \hat{P}_r (it: \{ \tilde{\kappa}_\nu\}) \exp \Big (- \frac {1} {2} \langle Dt, t \rangle \Big) \Big | dt 
$$
\begin {equation} \label {d95} = I^{\prime\prime}_1 + I^{\prime\prime}_2 + I^{\prime\prime}_3 + I^{\prime\prime}_4 + I^{\prime\prime}_5 .
\end {equation} 
Since $\|\mathbb{E}Y_j\|^s\leq(\rho_2(Y_j))^{\frac{s}{2}}\leq\rho_s(Y_j)$ and $\rho_1(Y_j) \leq 1$, then the following chain of trivial inequalities is valid for $ 1 \leq m \leq s-1,\;j=1\dots n$
$$
\rho_s (Q \theta_jZ_j) \leq \| Q\| ^{s}\rho_s (\theta_jZ_j)
  = \| Q\| ^{s}\rho_s (\theta_j (Y_j- \mathbb {E} Y_j)),
$$
\begin{equation} \label {d13}
\leq\| Q\| ^ {s} 2 ^ s \rho_s (\theta_jY_j)
\leq\| Q\| ^ {s} 2 ^ s \rho_ {s-m} (\theta_jX_j).
\end{equation}
Then by Lemma \ref {lemma2}, we obtain
$$ \det Q \leq \| Q ^ 2 \| ^ {\frac {k} {2}} \leq \Big (\frac {4} {3} \Big) ^ {\frac {k} {2 }},\;\;
\| Qt \| \geq \| D \| ^ {- \frac {1} {2}} \| t \| \geq \Big (\frac {4} {5} \Big) ^ {\frac {1} {2}} \| t \|,
$$
\begin {equation} \label {d96_1}
\Big\{\| Qt \| \leq  E_n^\prime \Big\} \subset \Big\{\| t \| \leq c_1 (k, k + 5) \min \Big\{\eta_ {k + 5} ^ {- \frac {1} {k + 5}}, \eta_ {k + 5} ^ {- \frac {1 } {k + 3}} \Big\} \Big\},
\end{equation}
where  $ Q^2=D^{-1} $. Also, take into account that for any $ s \geq4 $ from (\ref {d13}) it follows
\begin {equation} \label {d96}
\eta_s \leq 2 ^ s \| Q \| ^ s \sum_ {j = 1} ^ n \rho_5 (\theta_j X_j) = \| Q \| ^ s2 ^ s \rho_5 \leq 2 ^ s  (4/3) ^ {s/2} \rho_5.
\end{equation}
Substituting $ t = Qt $, $ s = k + 5 $ and using Lemma \ref {lemma3} \cite{Ay}, (\ref{d96_1}) and (\ref{d96}), we come to the fact that
$$ 
I_1^{\prime\prime} = \int \limits _ {\| Qt \| \leq E_n^\prime } \Big | D ^ \alpha \Big [\prod_ {j = 1} ^ n \varphi_1 (\theta_jQt) - \exp \Big (-\frac {  \| t \| ^ 2} {2} \Big) \sum_ {r = 0} ^ {k+2} \hat{P}_r (iQt, \{ \tilde{\kappa}_v\}) \Big] \det Q \Big | dt
$$
$$
\leq \int \limits _ {\| t \| \leq E_n} \Big (\frac {4} {3} \Big) ^ {\frac {k} {2}} \Big | D ^ \alpha \Big [\prod_ {j = 1} ^ n \varphi_1 (\theta_jQt) - \exp \Big (-\frac { \| t \| ^ 2} {2} \Big) \sum_ {r = 0} ^ {k+2} \hat{P}_r (iQt, \{ \tilde{\kappa}_v\}) \Big] \Big | dt
$$
$$
\leq \int \Big (\frac {4} {3} \Big) ^ {\frac {k} {2}} c_2 (k, k + 5) \eta_{k + 5} (\| t \| ^ {k + 5- | \alpha |} + \| t \| ^ {3 (k + 3) + | \alpha |}) \exp \Big (- \frac {1} {4} \| t \| ^ 2 \Big) dt
$$
\begin {equation} \label {d97}
\leq \hat {C_6} (\alpha, k) \eta_ {k + 5} \leq 2 ^ {k + 5}  (4/3) ^ {\frac {k + 5} {2}} \hat {C_6} (\alpha, k) \rho_5\leq C_5 (\alpha, k) \rho_5.
\end {equation} 
For the second integral we use inequality (\ref{d72_0}) from Lemma \ref{lemma6} and (\ref{d13})  and get
$$
I_2^{\prime\prime}
\leq \int \limits _ { E_n^\prime  \leq \| t \| \leq\frac{1} {16\rho_3} } c_5 (\alpha, k) (1+ \| t \| ^ {| \alpha |}) \exp \Big (- \frac {5} {48} \| t \| ^ 2 \Big) dt
$$
$$
\leq \int \limits (\| t \| ^ {- 1}  E_n^\prime  ) ^ {- k-5} c_5 (\alpha, k) (1+ \| t \| ^ {| \alpha |}) \exp \Big (- \frac {1} {48} \| t \| ^ 2 \Big) dt
$$
$$
\leq \hat {C} _7 (\alpha, k) \min \Big\{\eta_ {k + 5} ^ {- \frac {1} {k + 5}}, \eta_ {k + 5} ^ {- \frac {1} {k + 3}} \Big\} ^ {- k-5}
= \hat {C} _7 (\alpha, k) \eta_ {k + 5} \max \Big\{1, \eta_ {k + 5} ^ {\frac {2} {k+3}} \Big\}
$$
\begin {equation} \label {d98}
\leq \hat {C} _7 (\alpha, k) 2 ^ {k + 5} \Big (\frac {4} {3} \Big) ^ {\frac {k + 5} {2}} \rho_5 \Big(1+ \eta_ {k + 5} ^ {\frac {2} {k + 3}}\Big) \leq C_6 (\alpha, k) \rho_5.
\end {equation} 
To consider $I_3^{\prime\prime}$ we split the integral into two parts  
$I_3^{\prime\prime} =\mathbb{I}(\Omega_1)I_3^{\prime\prime}+\mathbb{I}(\Omega_2)I_3^{\prime\prime},$ where $$\Omega_1=\{16\rho_3>  \beta_3/\sqrt n\},\;\;\;\Omega_2=\{16\rho_3\leq\beta_3/\sqrt n\}.$$
At first, it was shown in Lemma
\ref{lemma6} $| D ^ \beta \varphi_j (\theta_j t)  | \leq \overline{c}_{18} (\alpha, k) \theta_j ^ 2 \rho_{2,j} \max \{1, \| t \| \}$, therefore
we get the following
\begin {equation} \label {d100}
I_3^{\prime\prime}
\leq\int \limits_{0 \leq \| t \| \leq n^{3/2} } \Big | D ^ \alpha \prod_ {j = 1} ^ n \varphi_1 (\theta_jt) \Big | dt
\leq \tilde{C}_8(\alpha, k)\int \limits_{0 \leq \| t \| \leq n^{3/2} }  dt\leq \tilde{C}_9(\alpha,k)n^{\frac{3k}{2}},
\end{equation}
then by inequality from Lemma 5.3.3 \cite{Bob} with  $r=16/33$ and (\ref{d100}) we have

$$\mathbb{E_\theta}\mathbb{I}(\Omega_1)I_3^{\prime\prime}\leq \exp(-n^{\frac{2}{3}})\tilde{C}_9(\alpha,k)n^{\frac{3k}{2}}$$
\begin {equation} \label {d101}
\leq \tilde{C}_{10}(\alpha,k)\exp(-\tilde{c}_1(k)n^{\frac{2}{3}})\leq \overline{C}_1(a,k)\beta_5/n^{3/2}.
\end{equation}
next recall that for identically distributed random vectors  $\beta_s=\rho_{s,1},\; s\geq1$, by inequality (\ref{d72}) from Lemma \ref{lemma7} we have
\begin {equation} \label {d102}
\mathbb {E}\mathbb{I}(\Omega_2)I_3^{\prime\prime}\leq
 \mathbb {E}\int \limits _ {\frac{\sqrt n}{\beta_3} \leq \| t \| \leq \frac{ n^{3/2}}{\beta_5}} \Big | D ^ \alpha \prod_ {j = 1} ^ n \varphi_1 (\theta_jt) \Big | dt \leq \overline{C}_2 (\alpha, k) \frac {\beta_5 } {n^{3/2}},
 \end{equation}
from  (\ref{d101}) and (\ref{d102}) it follows
\begin {equation} \label {d103}
\mathbb{E}_\theta I_3^{\prime\prime}\leq C_7 (\alpha, k) \frac {\beta_5 } {n^{3/2}}.
\end{equation}
In order to bound the  integrals $I_4^{\prime\prime}$ and $I_5^{\prime\prime}$, we use the  Lemma 6 \cite{Ay} and have ones
$$
I_4^{\prime\prime} \leq \int \limits _ { E_n^\prime \leq \| t \|} C_{11} (\alpha, k) (1+ \| t \| ^ {| \alpha |}) \exp \Big (- \frac {3} {8} \| t \| ^ 2 \Big ) dt
$$
\begin {equation} \label {d104}
\leq \int  (\| t \| ^ {- 1} E_n^\prime ) ^ {- k-5}  \overline{C}_{3} (\alpha, k) (1+ \| t \| ^ {| \alpha |}) \exp \Big (- \frac {3} {8} \| t \| ^ 2 \Big) dt
\leq C_8 (\alpha, k) \rho_5, 
\end {equation} 

\begin {equation} \label {d105}
I_5^{\prime\prime}
\leq \sum_ {r = 3} ^ {k+2} \overline{C}_{4} (\alpha, k, r) \rho_ {r + 2}(1+\rho_{2}^{r-1}) \leq C_9 (\alpha, k) \rho_5.
\end {equation} 
 Collecting results (\ref{d94}), (\ref{d94_1}), (\ref{d95}), (\ref{d97}), (\ref{d98}), (\ref{d104}) and (\ref{d105}), we get 
\begin{equation} \label {d106}
 I_3^\prime\leq C_9(k)\Big(\rho_5+I_3^{\prime\prime} +\frac{\beta_5}{n^{3/2}}\Big),
\end {equation}
finally, in case when $\rho_5\leq(8k)^{-1}$, we use (\ref{d88}), (\ref{d89}), (\ref{d91_1}),  (\ref{d106}) and have bound 
\begin {equation} \label {d107}
I\leq C^\prime(k)\Big(\rho_5+I_2 +I_3^{\prime\prime} +\frac{\beta_5}{n^{3/2}}\Big),
\end {equation}
Next in case if $\rho_5>(8k)^{-1}$, then there is obviously a constant $C^{\prime\prime}(k)$ such that
\begin {equation} \label {d108}
I\leq C^{\prime\prime}(k)\rho_5.
\end {equation}
Therefore, denoting sets $\Gamma_1=\{\rho_5\leq(8k)^{-1}\}$ and $\Gamma_2=\{\rho_5>(8k)^{-1}\}$ and using (\ref{d107}) and (\ref{d108}),  we conclude
\begin {equation} \label {d109}
I\leq  \mathbb{I}(\Gamma_1)I+ \mathbb{I}(\Gamma_2)I\leq \overline{C}(k)\Big(\rho_5+I_2 +I_3^{\prime\prime} +\frac{\beta_5}{n^{3/2}}\Big).
\end {equation}
To complete the proof of (\ref{d0}), we have to use (\ref{d91_1}), (\ref{d103}) and apply Lemma 5.3.1 \cite{Bob} with $p=5$ in (\ref{d109}). 
\end {proof}

\end{fulltext}

\end{document}